\documentclass{article}
\usepackage[%
journal=GGD,    %% Replace XXX by one of the above journal codes.
lang=american,   %% Change to `american' if you use American English.
]{ems-journal}

\newtheorem{theorem}{Theorem}[section]

\newtheorem{definition}[theorem]{Definition}

\newtheorem{lemma}[theorem]{Lemma}
\newtheorem{cor}[theorem]{Corollary}
\newtheorem{remark}[theorem]{Remark}

\newtheorem{claim}[theorem]{Claim}
\newtheorem{prop}[theorem]{Proposition}

\theoremstyle{remark}

\newcommand{\be}{\begin{enumerate}}
\newcommand{\ee}{\end{enumerate}}
\newcommand{\beq}{\begin{equation}}
\newcommand{\eeq}{\end{equation}}

\def\N{{\mathbb{N}}}
\def\Z{{\mathbb{Z}}}

\def\MA{{\mathbb{A}}}
\def\MB{{\mathbb{B}}}
\def\MC{{\mathbb{C}}}

\def\barx{\hat{x}}

{}
{}
{}

\numberwithin{equation}{section}

\begin{document}

%------
% Insert the title of your paper and (if necessary)
% a short title for the running head.
%------
\title{Groups elementarily equivalent to a finitely generated free metabelian group}
\titlemark{Groups elementarily equivalent to metabelian group}

%------

%%%% Pls fill in all fields for each author
%%%% Label the authors by their position in the authors' list using {}
%%%% If you published any math paper ever, you have an MR Author ID.
%  Please look it up in three easy (and free) steps:
% 1. copy the bibliographic data of any published paper (co-)authored by you in the search field at https://mathscinet.ams.org/mathscinet/freetools/mref
% 2. Hit your name in the search result
% 3. Find your MR Author ID in the first row, copy it in the \mrid{} field
%%%% If you have not created your ORCID yet, you may like to do it now, pls copy it in the field \orcid{}
%%%% Abbreviate first names for the running head

\emsauthor{1}{
	\givenname{Olga}
	\surname{Kharlampovich}
	\mrid{191704}
	\orcid{0000-0002-4891-5838}}{O.~Kharlampovich}
%%%% Repeat the same fields for each numbered author
\emsauthor{2}{
	\givenname{Alexei}
	\surname{Miasnikov}
	\mrid{670299}
	\orcid{0009-0006-7271-0011}}{A. Miasnikov}

%%%% Please provide detailed address info for each author
%%%% Use the same numbering as for \emsauthor above
%%%% Please look up the ROR ID of your institute here: https://ror.org
\Emsaffil{1}{
	\department{Department of Mathematics}
	\organisation{CUNY Graduate Center and Hunter College}
	%\rorid{01a2bcd34}
	\address{695 Park Ave}
	\zip{10065}
	\city{New York}
	\country{USA}
	\affemail{okharlampovich@gmail.com}}
%%%% Repeat the same fields for each numbered author
%%%% If some author has multiple affiliations, repeat the fields for each affiliation
%%%% Number the affiliations using {}
\Emsaffil{2}{
	\department{1}{Department of Mathematical Sciences}
	\organisation{1}{Stevens Institute}
	%\rorid{1}{}
	\address{1}{1 Castle Point Terrace}
	\zip{1}{07030}
	\city{1}{Hoboken}
    \country{1}{USA}
	\affemail{amiasnikov@gmail.com}}

%------
% Add MSC 2020 codes according to https://zbmath.org/classification/.
% A unique primary MSC code (in curly brackets) is mandatory,
% while secondary MSC codes (in square brackets) are optional.
%------
\classification{20A15}

%------
% Add a list of keywords.
%------
\keywords{metabelian group, first-order equivalence}

%------
% Insert your abstract.
%------
\begin{abstract}
We describe groups elementarily equivalent to a free metabelian group with $n$ generators. We also explore an exponentiation that naturally occurs in metabelian groups.
\end{abstract}

\maketitle

%------
% INSERT THE BODY OF THE PAPER HERE (except
% acknowledgments, funding info and bibliography)
%------
\tableofcontents

\section{Introduction}

In this paper we describe all groups elementarily equivalent to a free metabelian group $G$ of finite rank $> 1$. To do this, we first prove that $G$ is regularly bi-interpretable with $\Z$. Then following \cite{DM1}   we show that groups elementarily equivalent to $G$ are precisely the non-standard versions $G(\tilde \Z)$ of $G$, where $\tilde \Z \equiv \Z$,   and describe their algebraic structure. Along the way, we prove that the set of all bases of $G$ is absolutely definable in $G$. For this we provide a new characterization of bases of $G$ in terms of the normal forms and their coordinate functions which is interesting in its own right. Furthermore, this regular bi-interpretation of $G$ with $\Z$ is  strong and injective, which gives many interesting model-theoretic properties of $G$.  Thus $G$ is rich, that is, the first-order logic over $G$ is as expressive as the weak second-order logic over $G$ (see \cite{KMS}), $G$ admits elimination of imaginaries and the projective logical geometries over $G$ and $\Z$ form equivalent categories (see \cite{DM2}), etc. An important part of our characterization of the algebraic structure of non-standard models $G(\tilde \Z)$ comes from the theory of exponential groups (see \cite{MR1,MR2,AMN}). It turns out that every non-standard group $G(\tilde \Z)$ is an exponential $\tilde \Z$-group and we can describe the $\tilde \Z$-exponentiation in $G(\tilde \Z)$.

\section{Interpretability and bi-interpretability}
One can use the model-theoretic notion of interpretability and bi-interpretability to study structures
elementarily equivalent to a given one. In this paper, we are going to do this for free metabelian groups.
We recall here some precise definitions and several known facts that may not be very familiar to algebraists. 

A \emph{language} (or \emph{a signature}) $L$ is a triple $(Fun, Pr, C)$, where $Fun = \{f, \ldots \}$ is a  set of functional symbols $f$ coming together with their arities  $n_f \in \mathbb{N}$,  $Pr$ is  a set of relation (or predicate) symbols $Pr= \{P, \ldots \}$ coming together with their arities  $n_P \in \mathbb{N}$,  and a set of constant symbols $C = \{c, \ldots\}$. Sometimes we write $f(x_1, \ldots,x_n)$ or $P(x_1, \ldots,x_n)$ to show that $n_f = n$ or $n_P = n$.  Usually we denote variables by small letters $x,y,z, a,b, u,v, \ldots$, while the same symbols with bars $\bar x, \bar y,  \ldots$ denote tuples of the corresponding variables, say   $\bar x = (x_1, \ldots,x_n)$.  In this paper we always assume, if not said otherwise, that the languages we consider are finite. 
The following languages appear frequently throughout the text:  
the language of groups $\{ \cdot, ^{-1}, 1\}$,   where 1 is the constant symbol for the identity element,  $\cdot$ is the binary multiplication symbol, and  $^{-1}$ is the symbol of inversion; and the language of unitary  rings $\{+, \cdot, 0,1 \}$ with the standard symbols for addition, multiplication, the additive identity $0$, and the unity 1.

A structure in the language $L$ (an $L$-structure) with the base set $A$ is sometimes denoted by $\mathbb{A} = \langle A; L\rangle$ or simply by 
$\mathbb{A} = \langle A; f, \ldots, P, \ldots,c, \ldots \rangle$.  For a given structure $\mathbb{A}$ by $L(\mathbb{A})$ we denote the language of $\mathbb{A}$.   When the language $L$ is clear from the context, we follow the standard algebraic practice and denote the structure $\mathbb{A} = \langle A; L\rangle$ simply by $A$. 

Let $\mathbb{B} = \langle B ; L\rangle$ be a structure. A subset $A \subseteq B^n$ is called {\em definable}\index{deinable subset} in $\mathbb{B}$ if there is a formula $\varphi(x_1, \ldots,x_n)$ (without parameters) in $L(\mathbb{B})$ such that  $A = \{(b_1,\ldots,b_n) \in B^n \mid \mathbb{B} \models \varphi(b_1, \ldots,b_n)\}$. In this case we denote $A$ by $\varphi(B^n)$ or $\varphi(\MB)$ and  say that  \emph{$\varphi$ defines $A$} in $\mathbb{B}$.  Similarly, an operation $f$ on the subset  $A$ is definable in $\mathbb{B}$ if its graph is definable in $\mathbb{B}$.  A constant $c$ is definable if the relation $x=c$ is definable. An $n$-ary predicate $P(x_1,\ldots ,x_n)$ is definable in $\mathbb{B}$ if the set $\{(b_1,\ldots ,b_n)\in\mathbb{B}^n| P(b_1,\ldots ,b_n) \ { \rm is\  true}\}$
is definable in $\mathbb{B}$. 

\begin{definition} \label{de:interpretable} An algebraic  structure $\mathbb{A} = \langle A ;f, \ldots, P, \ldots, c, \ldots\rangle$  is absolutely interpretable (or $0$-interpretable)  in a structure $\mathbb{B}$  if there is a  subset $A^* \subseteq B^n$  definable in $\mathbb{B}$, an equivalence relation $\sim$ on $A^*$ definable in $\mathbb{B}$, operations  $f^*, \ldots, $ predicates $P^*, \ldots, $ and constants $c^*, \ldots, $ on the quotient set $A^*/{\sim}$ all interpretable in $\mathbb{B}$ such that the structure $\mathbb{A}^* = \langle A^*/{\sim}; f^*, \ldots, P^*, \ldots,c^*, \ldots, \rangle$ is isomorphic to $\mathbb{A}$.
 \end{definition}

More formally, an interpretation  of $\MA$ in $\MB$ is described  by the following set of formulas in the language $L(\MB)$
\[
\Gamma  = \{U_\Gamma(\bar x), E_\Gamma(\bar x_1, \bar x_2), Q_\Gamma(\bar x_1, \ldots,\bar x_{t_Q}) \mid Q  \in L(\MA)\}
\]
(here $\bar x$ and $\bar x_i$ are $n$-tuples of variables)  which  interpret $\mathbb{A}$ in $\mathbb{B}$ (as in the definition \ref{de:interpretable} above). Namely,  $U_\Gamma$  defines in $\mathbb{B}$ a subset  $A_\Gamma  = U_\Gamma(B^n)  \subseteq B^n$, $E_\Gamma$  defines in $\MB$ an  equivalence relation $\sim_\Gamma$ on $A_\Gamma$, and the formulas $Q_\Gamma$ define   functions $f_\Gamma$, predicates $P_\Gamma$, and constants $c_\Gamma$ that interpret the corresponding symbols from $L(\mathbb{A})$ on the quotient set $A_\Gamma/\sim_\Gamma$ in such a way that the $L$-structure $\Gamma(\MB) = \langle A_\Gamma/\sim_\Gamma; f_\Gamma, \ldots, P_\Gamma, \ldots, c_\Gamma, \ldots \rangle $ is isomorphic to $\MA$.  Note, that   we interpret a constant $c \in L(\MA)$ in the structure $\Gamma(\MB)$ by the $\sim_\Gamma$-equivalence  class of some tuple $\bar b_c \in A_\Gamma$ defined in $\MB$ by the formula $Q_c$. We write $\mathbb{A} \simeq \Gamma(\mathbb{B})$ if $\Gamma$ interprets $\MA$ in $\MB$ as described above and refer to $\Gamma$ as an {\em interpretation code}\index{interpretation code} or just {\em code}.  The number $n$ is called the dimension of $\Gamma$, denoted $n = dim\Gamma$. By $\mu_\Gamma$ we denote a a surjective map $A_\Gamma \to \MA$ (here $\MA = \langle A;L(\MA)\rangle$) that gives rise to an isomorphism   $\bar \mu_\Gamma: \Gamma(\MB) \to \mathbb{A}$.  We refer to  this map $\mu_\Gamma$ as the \emph{the coordinate map}\index{coordinate map} of the interpretation $\Gamma$. When the formula $E_\Gamma$ defines the identity relation $(x_1=x^\prime_1)\wedge\ldots\wedge(x_n=x^\prime_n)$, the surjection $\mu_\Gamma$ is injective, in which case, $\Gamma(\MB)$ is called an {\em injective interpretation}. Finally, notation $\MA\stackrel{\Gamma}{\rightsquigarrow}\MB$   means that  $\MA$ is interpretable in $\MB$ by the code $\Gamma$. 

More generally, the formulas that interpret $\mathbb{A}$ in $\mathbb{B}$ may contain elements from $\mathbb{B}$ that are not in the language $L(\mathbb{B})$, i.e., some parameters, say $p_1, \ldots,p_k \in B$.  In this case we assume that all the formulas from the code $\Gamma$ have a tuple  of extra variables $\bar y = (y_1, \ldots,y_k)$  for parameters in $\MB$: 
\begin{equation} \label{eq:code}
\Gamma =  \{U_\Gamma(\bar x,\bar y), E_\Gamma(\bar x_1, \bar x_2,\bar y), Q_\Gamma(\bar x_1, \ldots,\bar x_{t_Q},\bar y) \mid Q \in L(\mathbb{A})\}
\end{equation}
so that after the assignment $y_1 \to p_1, \ldots,y_k \to p_k$ the code  interprets $\mathbb{A}$ in $\mathbb{B}$.  In this event we write $\mathbb{A} \simeq \Gamma(\mathbb{B},\bar p)$ (here $\bar p = (p_1, \ldots,p_k)$), and say that $\mathbb{A}$ is interpretable in $\mathbb{B}$  by the code $\Gamma$ with parameters $\bar p$. In the case when $\bar p = \emptyset$ one gets again the absolute interpretability. Sometimes, it is convenient to consider interpretations $\MA \simeq \Gamma(\MB,\bar p)$  together with their coordinate maps, i.e., as triples $(\Gamma,\bar p,\mu_\Gamma)$.

We say that a structure $\mathbb{A}$ is interpreted in a given structure $\mathbb{B}$ \emph{uniformly}\index{uniform interpretation} with respect to a subset $D \subseteq B^k$ if  there is a code $\Gamma$ such that $\mathbb{A} \simeq \Gamma(\mathbb{B},\bar p)$ for every tuple of parameters $\bar p \in D$. If $\mathbb{A}$ is interpreted in  $\MB$ uniformly with respect to a 0-definable subset $D \subseteq B^k$ then we say that $\MA$ is \emph{regularly interpretable}\index{regular interpretation} in $\MB$ and write in this case $\MA \simeq \Gamma(\MB,\varphi)$, provided  $D$ is defined by $\varphi$ in $\MB$. 
This notion appeared first in \cite[Section 1.1]{Miasn} or \cite{Miasn1}, and it is similar to the notion of {\em interpretability with definable parameters} \cite[Remark 5, p.215]{Hodges}. Note that the absolute interpretability is a particular case of the regular interpretability where the set  $D$ is empty.

%\subsection{Composition of %interpretations}

It is known that the relation $\MA\rightsquigarrow\MB$  is transitive on algebraic structures (see, for example,~\cite{DM1, Hodges}). The proof of this fact is based on the notion of $\Gamma$-translation and composition of codes, which we present now. 

Let $\Gamma$ be the code (\ref{eq:code}). Then for    any formula $\varphi(x_1,\ldots,x_m)$ in the language $L(\MA)$ there is  a formula $\varphi_\Gamma(\bar x_1,\ldots,\bar x_m, \bar y)$ in the language $L(\MB)$, the $\Gamma$-translation of $\varphi$,  such that if $\MA\simeq\Gamma(\MB,\bar p)$, then for any coordinate map $\mu_\Gamma\colon A_\Gamma\to A$ one has
$$
\MA\models \varphi(a_1,\ldots, a_m) \iff \MB\models\varphi_\Gamma(\mu_\Gamma^{-1}(a_1),\ldots,\mu^{-1}_\Gamma(a_m),\bar p)
$$
for any elements $a_i\in A$ (see~\cite{DM1, Hodges}). Here $\mu_\Gamma^{-1}(a_i)$ means an arbitrary preimage of $a_i$ under $\mu_\Gamma$.   Furthermore,  for any elements $\bar b_i\in B^n$ if $\MB\models\varphi_\Gamma(\bar b_1,\ldots,\bar b_m,\bar p)$ then  $\bar b_i\in \mu_\Gamma^{-1}(a_i)$ for some $a_i\in A$ with $\MA\models \varphi(a_1,\ldots, a_m)$.

\begin{definition} \label{def:code-composition}
Let $\MA, \MB$, $\MC$ be algebraic structures. Consider codes 
$$\Gamma =  \{U_\Gamma(\bar x,\bar y), E_\Gamma(\bar x, \bar x^\prime,\bar y), Q_\Gamma(\bar x_1, \ldots,\bar x_{t_Q},\bar y) \mid Q \in L(\MA)\}$$
and 
$$
\Delta=\{U_\Delta(\bar u,\bar z), E_\Delta(\bar u, \bar u^\prime,\bar z), Q_\Delta(\bar u_1, \ldots,\bar u_{t_Q},\bar z) \mid Q \in L(\MB)\}
$$
which consist of $L(\MC)$-formulas. Then the {\em composition} of the codes $\Gamma$ and $\Delta$ is the code
$$
\Gamma \circ \Delta = \{U_{\Gamma\circ\Delta}, E_{\Gamma\circ\Delta}, Q_{\Gamma\circ\Delta} \mid Q \in L(\MA)\}=\{(U_\Gamma)_\Delta, (E_\Gamma)_\Delta, (Q_\Gamma)_\Delta \mid Q \in L(\MA)\}.
$$
\end{definition}

The following is an important technical result on the transitivity of interpretations.

\begin{lemma}[\cite{DM1}]\label{le:int-transitivity}
Let $\MA=\langle A;L(\MA)\rangle, \MB=\langle B;L(\MB)\rangle$ and $\MC=\langle C;L(\MC)\rangle$ be algebraic structures and $\Gamma,\Delta$ be codes as above. If $\MA\stackrel{\Gamma}{\rightsquigarrow}\MB$ and 
$\MB\stackrel{\Delta}{\rightsquigarrow}\MC$ then $\MA\stackrel{\Gamma\circ\Delta}{\rightsquigarrow}\MC$.

Furthermore, the following conditions hold:
\begin{enumerate}[label=\arabic*)]
    \item If $\bar p,\bar q$ are parameters and $\mu_\Gamma, \mu_\Delta$ are  coordinate maps of interpretations $\Gamma,\Delta$ then $(\bar{\bar p},\bar q)$, where $\bar{\bar p} \in \mu_\Delta^{-1}(\bar p)$, are parameters for $\Gamma\circ\Delta$; 
    
    \item $\mu_\Gamma\circ\mu_\Delta=\mu_{\Gamma}\circ\mu^n_{\Delta}\big|_{U_{\Gamma\circ\Delta}(\MC, (\bar{\bar p},\bar q))}$ is a coordinate map of the interpretation $\MA\simeq \Gamma\circ\Delta(\MC,(\bar{\bar p},\bar q))$ and any coordinate map $\mu_{\Gamma\circ\Delta}\colon U_{\Gamma\circ\Delta}(\MC, (\bar{\bar p},\bar q))\to A$ has a form $\mu_{\Gamma1}\circ\mu_\Delta$ for a suitable coordinate map $\mu_{\Gamma1}$ of the interpretation $\MA\simeq\Gamma(\MB,\bar p)$, provided $\mu_\Delta$ is fixed. 
\end{enumerate}
\end{lemma}

Observe, that composition of absolute (regular) interpretations is absolute (regular).

Now we discuss a very strong version of  mutual interpretability of two structures, so-called {\em bi-interpretability}.

\begin{definition}\label{def:bi}
Algebraic structures $\MA$ and $\MB$ are called {\em strongly bi-interpretable} (with parameters) in each other, if 
\begin{enumerate}[label=\arabic*)]
\item there exists an interpretation $(\Gamma,\bar p,\mu_\Gamma)$ of $\MA$ into $\MB$ and an interpretation $(\Delta,\bar q,\mu_\Delta)$ of $\MB$ into $\MA$, so the algebraic structures $\Gamma\circ\Delta(\MA,(\bar{\bar p},\bar q))$ and  $\Delta\circ\Gamma(\MB,(\bar{\bar q},\bar p))$  are uniquely defined and $\Gamma\circ\Delta(\MA,(\bar{\bar p},\bar q))$ is isomorphic to $\MA$, while $\Delta\circ\Gamma(\MB,(\bar{\bar q},\bar p))$ is isomorphic to $\MB$; 
\item the composition $\mu_\Gamma\circ\mu_\Delta\colon U_{\Gamma \circ \Delta}(\MA,(\bar{\bar p},\bar q)) \to A$ is definable in $\MA$ and the composition $\mu_\Delta\circ\mu_\Gamma\colon U_{\Delta \circ \Gamma}(\MB,(\bar{\bar q},\bar p)) \to B$ is definable in $\MB$.
\end{enumerate}
In this case, we additionally say that
$\MA$ and $\MB$ are {\em strongly injectively bi-interpretable}, if the interpretations $\Gamma$ and $\Delta$ are injective.  
\end{definition}

Note that there is another slightly different notion of {\em bi-interpretation}, which for contrast we sometimes call a {\em weak bi-interpretation}, where in the above definition the condition 2) that requires definability of the maps $\mu_\Gamma\circ\mu_\Delta$ and $\mu_\Delta\circ\mu_\Gamma$ is replaced by a weaker one that requires definability of some coordinate maps $A_{\Gamma \circ \Delta} \to A$ and $B_{\Delta \circ \Gamma} \to B$.  Often, authors do not even mention the difference, implicitly assuming either one or another. To be precise, we endorse these two notions explicitly. Observe that the bi-interpretation defined in the books~\cite{Hodges} and~\cite{KMS} is weak, but in the paper~\cite{AKNS} it is strong. There are many interesting applications of strong bi-interpretations which we cannot derive from the weak ones.

Algebraic structures  $\MA$ and $\MB$ are called {\em 0-bi-interpretable} or {\em absolutely  bi-interpretable}\index{absolulutely  by-interpretable} (strongly) in each other if in the definition above the tuples of parameters  $p$ and $q$ are empty.  

Unfortunately, 0-bi-interpretability is rather rare. Indeed,  if $\MA$ and $\MB$ are 0-bi-interpretable in each other then their groups of automorphisms are isomorphic \cite{Hodges}. Bi-interpretability with parameters occurs much more often, but it gives much less for applications, in particular, it is not applicable for first-order classification problems (addition of  constants changes the language).

Fortunately, there is a notion of regular bi-interpretability, which is less restrictive, occurs more often, and which  enjoys many properties of 0-bi-iterpretability.

\begin{definition}\label{def:reg}
Algebraic structures $\MA$ and $\MB$ are called {\em regularly bi-interpretable}, if  
\begin{enumerate}[label=\arabic*)]
\item there exist a regular interpretation $(\Gamma,\varphi)$ of $\MA$ in $\MB$ and a regular interpretation $(\Delta,\psi)$ of $\MB$ in $\MA$;
\item there exists formula $\theta_\MA(\bar u, x, \bar r)$ in $L(\MA)$, where $|\bar u|={\dim\Gamma\cdot\dim\Delta}$, $|\bar r|={\dim_{par}\Gamma\circ\Delta}$, such that for any tuple $\bar r_0\in \varphi_\Delta\wedge \psi(\MA)$ the formula $\theta_\MA(\bar u, x, \bar r_0)$ defines some coordinate map $U_{\Gamma\circ\Delta}(\MA,\bar r_0)\to A$;
\item there exists formula $\theta_\MB(\bar u, x, \bar t)$ in $L(\MB)$, where $|\bar u|={\dim\Gamma\cdot\dim\Delta}$, $|\bar t|={\dim_{par}\Delta\circ\Gamma}$, such that for any tuple $\bar t_0\in \psi_\Gamma\wedge \varphi(\MB)$ the formula $\theta_\MB(\bar u, x, \bar t_0)$ defines some coordinate map $U_{\Delta\circ\Gamma}(\MB,\bar t_0)\to B$. 
\end{enumerate}
Regular bi-interpretation is called {\em injective} if both $(\Gamma,\varphi)$ and $(\Delta,\psi)$ are injective.
\end{definition}

\begin{definition}\label{def:st_reg}
We say that $\MA$ and $\MB$ are {\em strongly regularly bi-interpretable}, if they are regularly bi-interpretable, i.\,e., 1)--3) hold, and additionally 
\begin{enumerate}%[label=\arabic*), resume]
\item[4)] for any 
pair of parameters $(\bar p, \bar q)$, $\bar p\in \varphi(\MB)$, $\bar q\in \psi(\MA)$, there exists a pair of coordinate maps $(\mu_\Gamma,\mu_\Delta)$ for interpretations $(\Gamma,\bar p)$ and $(\Delta,\bar q)$, such that for any $\bar r_0=(\bar{\bar p},\bar q)$, $\bar{\bar p}\in \mu^{-1}_\Delta(\bar p)$, and $\bar t_0=(\bar{\bar q},\bar p)$, $\bar{\bar q}\in \mu^{-1}_\Gamma(\bar q)$, the coordinate maps ${\mu_\Gamma\circ\mu_\Delta\colon} U_{\Gamma\circ\Delta}(\MA,\bar r_0)\to A$ and ${\mu_\Delta\circ\mu_\Gamma\colon} U_{\Delta\circ\Gamma}(\MB,\bar t_0)\to B$ are defined in $\MA$ and $\MB$ correspondingly by the formulas $\theta_\MA(\bar u, x, \bar r_0)$ and $\theta_\MB(\bar u, x, \bar t_0)$.
\end{enumerate}
\end{definition}

%%%%%%%%%%%%%%%%%%%%%

\begin{theorem}[\cite{DM1,MN}] \label{th:main}
Let $\mathbb A$ and $\mathbb B$ be regularly bi-interpretable in each other, so $\mathbb A \cong \Gamma(\mathbb B,\varphi)$ and $\mathbb B\cong\Delta(\mathbb A,\psi)$. Then
\begin{enumerate}[1)]
\item For any $\widetilde{\mathbb B}\equiv\mathbb B$ the code $(\Gamma,\varphi)$ regularly interprets a structure $\widetilde{\mathbb A} \simeq \Gamma(\widetilde{\mathbb B},\varphi)$ in $\widetilde{\mathbb B}$ such that  $\widetilde{\mathbb A} \equiv \mathbb A$;
\item Every $L(\mathbb A)$-structure $\widetilde{\mathbb A}$ elementarily equivalent to $\mathbb A$ is isomorphic to $\Gamma(\widetilde{\mathbb B},\varphi)$ for some $\widetilde{\mathbb B}\equiv\mathbb B$;
\item For any $\mathbb B_1\equiv\mathbb B\equiv\mathbb B_2$ one has 
\[
\Gamma(\mathbb B_1,\varphi) \cong \Gamma(\mathbb B_2,\varphi) \iff \mathbb B_1 \cong \mathbb B_2.
\]
\end{enumerate}
\end{theorem}

In fact, one does not need bi-interpretability of $\mathbb A$ and $\mathbb B$ in Theorem \ref{th:main} it suffices to have a weaker condition, that $\mathbb A$ is regularly invertibly interpretable in  $\mathbb B$ (see \cite{DM1,MN}). However, for other applications, bi-interpretability is required.

Elimination of imaginaries plays an important part in model theory (see \cite{Hodges} for details). This involves imaginaries \emph{without parameters}.

\begin{definition}
We say that an algebraic structure $\MB$ has {\em uniform elimination of imaginaries with parameters}, if the algebraic structure $\MB_B$, obtained from $\MB$ by adding all elements of $\MB$ to the language as constants, has uniform elimination of imaginaries without parameters.
\end{definition}

\begin{theorem} \cite{DM1}
Let $\MA$ and $\MB$ be strongly bi-interpretable with parameters in each other, so $\MA \simeq \Gamma(\MB,\bar p)$ and $\MB\simeq\Delta(\MA,\bar q)$, and $\Delta$ is injective. If $\MB$ has uniform elimination of imaginaries with parameters, then $\MA$ has uniform elimination of imaginaries with parameters. 
\end{theorem}

Note, that $\Z$ has uniform elimination of imaginaries (with or without parameters). Hence, all structures strongly injectively bi-interpretable with $\Z$ (with parameters) enjoy uniform elimination of imaginaries with parameters.

 \section{By-interpretation of a free metabelian group with $\mathbb Z$}
 \label{se:free-metabelian}
It was announced in \cite{Khelif} and  proved in \cite{KMS} that a free metabelian group $G$ of finite rank $n \geq 2$ is prime, atomic, homogeneous, and QFA. Furthermore, it was shown in \cite{KMS} that $G$ is rich which implies a host of model-theoretic results for $G$. This was done by proving the existence of bi-interpretation of $G$ with $\Z$.  Here we improve on our   bi-interpretation from \cite{KMS} and provide  new applications of this  in the next section.
  
 Throughout this section we denote by $G$ a free metabelain group
of rank $n\geq 2$.

\subsection{Preliminaries for metabelian groups}\label{se:3.1}

In this section we introduce notation and describe  some   results  that we need in the sequel.  

Let $G' = [G,G]$ be the commutant of $G$ and $G_m$  the $m$'th term of the lower central series of $G$. For a subset $A \subseteq G$ denote by   $\langle A \rangle$ the subgroup generated by $A \subseteq G$ and by $C_G(A)$ -  the centralizer of $A$ in $G$. If  $g, h \in G$  then  $[g,h] = g^{-1}h^{-1}gh$ is the commutator of $g$ and $h$, and $g^h = h^{-1}gh$ is the conjugate of $g$ by $h$.  The maximal root of an element $g \in G$ is an element $g_0 \in G$ such that $g_0$ is not a proper power in $G$ and $g \in \langle g_0\rangle$. Note that the maximal roots of the elements in $G$ exist and they are unique.

We will be using the following standard commutator identities that hold in every group for any elements $a,b,c$.

  \begin{equation} \label{eq:comm-identities-1}
      [a,b]^{-1} = [b,a], \ [a^{-1},b] = [b,a]^{a^{-1}},
      \end{equation}
\begin{equation} \label{eq:comm-identities-2}
     [ab,c] = [a,c]^b[b,c], \ [a,bc] = [a,c][a,b]^c. 
   \end{equation}

 In \cite{Mal1} Mal'cev obtained a description of centralizers of elements in free solvable groups, in particular, in free metabelian groups $G$ centralizers are as follows. 
 \begin{lemma} \cite{Mal1}
Let $g \in G, g \neq 1$. Then 
\begin{itemize}
\item [1)] if $g \in G'$ then $C_G(g) = G'$
\item [2)] if $g \not \in G'$ then $C_G(g) = \langle g_0\rangle$, where $g_0$ is the unique maximal root of $g$. 
\end{itemize}
 \end{lemma}

\medskip \noindent
 Let $v \in G\smallsetminus G'$. Define a map $\lambda_v:G' \to G'$ such that  for $c \in G'$ 
$ \lambda_v(c) = [v,c]$.  Then the map $\lambda_v$ is a homomorphism.  Indeed, using the second commutator identity in (\ref{eq:comm-identities-2})
 one has for $c_1,c_2 \in G'$ 
\begin{equation}\label{eq:1.2}
[v,c_1c_2] = [v,c_2][v,c_1]^{c_2} = [v,c_2][v,c_1] =  [v,c_1][v,c_2],
\end{equation}
as claimed. Similarly, in the notation above,  the map $\mu_v: c \to [c,v]$ is a homomorphism $\mu_v:G' \to G'$.

\medskip \noindent
 Let $v \in G\smallsetminus G'$ and $d \in G'$. Then for any $k \in \Z$ there exists $c \in G'$ such that:
$$
(vd)^k = v^kd^k[c,v].
$$
We prove first, by induction on $k$,  that 
$$
d^kv = vd^k[c,v]
$$ 
for  some $c \in G'$. Indeed, for $k = 1$ one has the standard equality $dv = vd[d,v]$. Now 
$$
d^{k+1}v = d^kdv = d^kvd[d,v] = vd^k[c_1,v]d[d,v] = vd^{k+1}[c_1,v][d,v]=vd^{k+1}[c_1d,v],
$$
the last equality comes from the property~\eqref{eq:1.2}, that the map $\mu_v$ is a homomorphism on $G'$.

Now one can finish the claim by induction on $k$ as follows (here elements $c_i \in G'$ appear as the result of application of the induction step and the claim above):
$$
(vd)^{k+1} = (vd)^kvd  = v^kd^k[c_2,v]vd = v^kd^kv[c_2,v][[c_2,v],v]d = 
$$
$$
= v^kvd^k[c_3,v][c_2,v][[c_2,v],v]d = v^{k+1}d^{k+1}[c_3c_2[c_2,v],v] =  v^{k+1}d^{k+1}[c,v]
$$
where the second to last equality comes again from the property~\eqref{eq:1.2}, and $c = c_3c_2[c_2,v]$. This proves the claim.

\medskip \noindent

\begin{lemma} \label{le:identity} \cite[Lemma 4.23]{KMS}
Let $d \in G'$. If for any $v \in G\smallsetminus G'$ there exists $c \in G'$ such that 
$d = [c,v]$, then $d = 1$.
\end{lemma}

\subsection{Normal forms of elements}
 \label{subsec:2.5}   
  
 Fix a finite set of generators $X = \{x_1, \ldots,x_n\}$.   
 
 We denote by $\bar{} :G \to G/G'$ the canonical epimorphism $g \to \bar g = gG'$ from the group $G$ onto its abelianization $\bar G = G/G'$. Put  $a_1 = \bar x_1, \ldots, a_n = \bar x_n$.
 The group $G$ acts by conjugation on $G'$, which gives an action of the abelianization $\bar G$ on $G'$.  This action extends linearly to an action of the group ring $\Z\bar G$ on $G'$ and turns $G'$ into a $\Z\bar G$-module. Since  $a_1, \ldots,a_n$ generate  $\bar G$, then then the identity map $a_i \to a_i, i = 1, \ldots,n,$ extends to  an epimorphism $\eta:A \to \Z\bar G$ from the  ring of Laurent polynomials  $A = \Z[a_1,a_1^{-1}, \ldots,a_n,a_n^{-1}]$ onto $\Z\bar G$, which provides an action of $A$ onto $G'$ and turns $G'$ into $A$-module. If $a_1, \ldots,a_n$ forms a basis of   $G'$, then $\eta$ is an isomorphism and we can identify $\Z\bar G$ with $A$. For the action of $A$ on $G'$ we use exponential notation, that is, for $u\in G'$ and $a \in A$ we denote by $u^a$ the result of the action of $a$ on $u$.  We use this notation throughout the paper.   
 
 Using the standard commutator identities, which hold in every group,
one can write every  element $u \in G'$ in the form 
\begin{equation} \label{eq:uncollected}
    u = \Pi_{1\leq j < i \leq n}[x_i,x_j]^{Q_{i,j}},
\end{equation}
 where $Q_{i,j}$ are Laurent polynomials from $A$. Hence the set of commutators
$$C_X = \{ [x_i,x_j] \mid 1\leq j < i \leq n\}$$
generates $G'$ as an $A$-module. Note, that  $G$, as well as any metabelian group, satisfies the Jacobi identity, i.e, for every $u, v, w \in G$ 
\begin{equation} \label{eq:Jacobi}
  [u,v,w][v,w,u][w,u,v] = 1.  
\end{equation}
In particular,  for $u = x_i, v = x_j, w = x_k$ one gets (in the module notation)
$$
[x_i,x_j]^{a_k-1}[x_j,x_k]^{a_i-1}[x_k,x_i]^{a_j-1} = 1,
$$
hence 
\begin{equation} \label{eq:xi-xj-xk}
    [x_i,x_j]^{a_k-1} = [x_k,x_j]^{a_i-1}[x_k,x_i]^{1-a_j},
\end{equation}
so $C_X$ is not a free generating set of the module $G'$. However, there are nice normal forms of elements of the module $G'$ (see \cite{MRom}, \cite{KMS}). To deal with normal forms in $G$ we need a few preliminary results.

 \begin{remark}
     Let $1 \neq \bar g \in \bar G$ and $\delta \in \Z$. Then $\bar g -1$ divides ${\bar g}^\delta -1 $ in the ring $\Z(\bar G)$, i.e., the element $\frac{{\bar g}^\delta -1}{\bar g -1}$ is uniquely defined in $\Z(\bar G)$. 
    Indeed, if $\delta > 0$ then 
     $$
     {\bar g}^\delta -1 = (\bar g -1)({\bar g}^{\delta-1} + \ldots + \bar g +1 ).
     $$
     If $\delta <0$ then ${\bar g}^\delta = ({\bar g}^{-1})^{|\delta|}$ and the formula above applies.  
 \end{remark}

The following generalizes the equality (\ref{eq:xi-xj-xk}).
\begin{lemma} \label{le:general-Jacobi}
Let $1\leq j < i < k \leq n$ and $\delta \in \Z$.
Then
\begin{equation} \label{eq:general-Jacoby}
    [x_i,x_j]^{a_k^{\delta} -1} = [x_k,x_j]^{(a_i-1)\frac{a_k^\delta -1}{a_k-1}}[x_k,x_i]^{(1-a_j)\frac{a_k^\delta-1}{a_k-1}}
\end{equation}
\end{lemma}
\begin{proof}
    Note that
    $$
    [x_i,x_j]^{a_k^{\delta} -1} = [x_i,x_j]^{(a_k-1)\frac{a_k^\delta -1}{a_k-1}}
    $$
Now we apply (\ref{eq:xi-xj-xk}) to $[x_i,x_j]^{(a_k-1)}$ and multiply the result (in the module $G'$) by $\frac{a_k^\delta -1}{a_k-1}$.
\end{proof}

\begin{lemma} \label{le:power-in-comm} Let $z,g \in G \smallsetminus G'$ and $\gamma, \delta \in \Z$ then  
\begin{enumerate}
  \item [1)]   $[z,g^{\delta}]=[z,g]^{\frac{{\bar g}^{\delta}-1}{\bar g -1}},$
  \item [2)] $[z^\gamma,g^{\delta}]=[z,g]^{(\frac{{\bar z}^{\gamma}-1}{\bar z -1}) (\frac{{\bar g}^{\delta}-1}{\bar g -1})}.$
\end{enumerate}

\end{lemma}
\begin{proof} The Jacobi identity 
$$
[z,g^\delta,g][g^\delta,g,z] [g,z,g^\delta] = 1
$$
reduces to 
$$
[z,g^\delta,g] = [g,z,g^\delta]^{-1},
$$
since $[g^\delta,g,z] = 1$.
After rewriting it in the module notation we get
$$
[z,g^\delta]^{\bar g -1} = [z,g]^{{\bar g}^\delta -1.}
$$
This proves 1).  2) follows from 1). 

\end{proof}

\begin{prop} \label{nf-b} Let $X = \{x_1, \ldots,x_n\}$ be a generating set of $G$ and $a_1 = \bar x_1, \ldots, a_n = \bar x_n$. Then every element $u \in G'$ can be presented as the following product
\begin{equation} \label{eq:collected}
    u = \Pi_{1\leq j<i\leq n}[x_i,x_j]^{\beta_{ij}(a_1,\ldots,a_i)},
\end{equation}
where $\beta_{ij}(a_1,\ldots,a_i) \in \Z[a_1,a_1^{-1}, \ldots,a_i,a_i^{-1}]$.\end{prop}
\begin{proof}
We showed above that every element $u \in G'$ can be written in the form (\ref{eq:uncollected}):
$$
 u = \Pi_{1\leq j < i \leq n}[x_i,x_j]^{Q_{i,j}}.
$$
where $Q_{i,j}$ are Laurent polynomials from $A$. Note that this can be done algorithmically. Now we describe a collecting process that transforms products in the form (\ref{eq:uncollected}) to products of the form (\ref{eq:collected}), which we term \emph{collected forms}. Since $G'$ is commutative, it suffices to show how to collect an element $[x_i,x_j]^Q$, where $Q \in A$. Similarly, since $Q$ is a sum of the type $\Sigma_i \gamma_iM_i$, where $M_i \in \bar G$ and $\gamma_i \in \Z$, it suffices to collect $[x_i,x_j]^M$, where $M \in \bar G$. Decompose $M$ into a product $M = M_1a_k^\delta M_2$, where $M_1 \in \langle a_1, \ldots,a_i\rangle$, $k > i$, $\delta \in \Z$, and $M_2 \in \langle a_{k+1}, \ldots,a_n\rangle$. Note that $[x_i,x_j]^{M_1}$ is collected. To collect $[x_i,x_j]^{a_k^\delta}$ write it as 
 $$
 [x_i,x_j]^{a_k^\delta} = [x_i,x_j]^{a_k^\delta-1}[x_i,x_j]
 $$  
and apply (\ref{eq:general-Jacoby}) from Lemma \ref{le:general-Jacobi}.
This results in a collected form $w$, where 
$$
w= \Pi_{1\leq j<i\leq k}[x_i,x_j]^{f_{ij}(a_1,\ldots,a_i)},
$$
for some  $f_{ij}(a_1,\ldots,a_i) \in \Z[a_1,a_1^{-1}, \ldots,a_i,a_i^{-1}]$. 
Note that 
$$
w^{M_1} = \Pi_{1\leq j<i\leq k}[x_i,x_j]^{f_{ij}(a_1,\ldots,a_i)M_1}
$$
is also collected. Since $[x_i,x_j]^M = (w^{M_1})^{M_2}$ the argument above shows that to collect $[x_i,x_j]^M$ it suffices to collect elements of the type $[x_i,x_j]^{M_2}$, where  $1\leq j<i\leq k$.  Now we can repeat the collecting process above. This shows that every element $u \in G'$ can be written in the form (\ref{eq:collected}). 
\end{proof}

\begin{cor} \label{co:forms} Let $X = \{x_1, \ldots,x_n\}$ be a generating set of $G$. Then every element $g \in G$ can be presented as the following product
\begin{equation} \label{eq:normal-form}
    g = x_1^{\gamma_1} \ldots x_n^{\gamma_n}\Pi_{1\leq j<i\leq n}[x_i,x_j]^{\beta_{ij}(a_1,\ldots,a_i)},
\end{equation}
where $\gamma_i \in \Z, 1\leq i\leq n$, $\beta_{ij}(a_1,\ldots,a_i) \in \Z[a_1,a_1^{-1}, \ldots,a_i,a_i^{-1}]$.
\end{cor}

\begin{prop} \label{nf-2} Let $X = \{x_1, \ldots,x_n\}$ be a basis of $G$ as a free metabelian group. Then every element $u \in G'$ can be uniquely presented as the following product
\begin{equation} \label{eq:collected-2}
    u = \Pi_{1\leq j<i\leq n}[x_i,x_j]^{\beta_{ij}(a_1,\ldots,a_i)},
\end{equation}
where $\beta_{ij}(a_1,\ldots,a_i) \in \Z[a_1,a_1^{-1}, \ldots,a_i,a_i^{-1}]$.\end{prop}
\begin{proof}
Observe that the images of $X$ in $\bar G$ form a basis of $\bar G$, hence $\Z\bar G \simeq A$. By Proposition \ref{nf-b} every element $u \in G'$ has some  decomposition of the form (\ref{eq:collected-2}). To prove uniqueness of the forms (\ref{eq:collected}) we use induction on $n$. Let  
$$
u = \Pi_{1\leq j<i\leq n}[x_i,x_j]^{\beta_{ij}(a_1,\ldots,a_i)},
$$
where $\beta_{ij}(a_1,\ldots,a_i) \in \Z[a_1,a_1^{-1}, \ldots,a_i,a_i^{-1}]$. Assume that $u = 1$. We need to show that $\beta_{ij} = 0$ for all $1\leq j<i\leq n$,

For $n = 2$ the $\Z[a_1^{\pm 1}, a_2^{\pm 1}]$-module $G'$ is free with basis $[x_2,x_1]$ (see \cite{Bahmuth}), so the result follows.

For $n > 2$ consider the canonical  epimorphism $\mu_n: G \to H = G/ncl(x_n)$, where $ncl(x_n)$ is the normal closure of $x_n$ in $G$, so $\mu_n(x_i) = x_i, 1\leq i < n$, $\mu_n(x_n) = 1$. Note that $H$ is a free metabelian group of rank $n-1$. Clearly,
$$
\mu_n(u) = \Pi_{1\leq j<i\leq n-1}[x_i,x_j]^{\beta_{ij}(a_1,\ldots,a_i)}
$$
 Hence by induction $\beta_{ij} = 0$ for $1\leq j<i\leq n-1$. Therefore, 
 $$
u = [x_n,x_1]^{\beta_{n1}(a_1,\ldots,a_n)} \ldots [x_n,x_{n-1}]^{\beta_{n,n-1}(a_1,\ldots,a_n)}.
$$
For an integer $N >0$ consider a homomorphism $\lambda_N: G \to K = \langle x_2, \ldots,x_n\rangle \leq G$ such that $\lambda_N(x_1) = x_2^N, \lambda_N(x_i) = x_i$ for $2\leq i \leq n$. Clearly, $\lambda_N$ induces the corresponding endomorphism on $\bar G$, hence on the ring $A$ and on the $A$-module $G'$. We continue to denote it by $\lambda_N$. Note that $K$ is a free metabelian group or rank $n-1$.

If $\bar g = a_1^{\delta_1} \ldots a_n^{\delta_n}$ then $\lambda_N(\bar g) = a_2^{\delta_1N+\delta_2}a_3^{\delta_3} \ldots a_n^{\delta_n}$. One can chose a large enough $N$ such that $\lambda_n$ is injective on all the monomials that occur in $\beta_{nj}$, $j = 1, \ldots,n-1$. This implies that if $\beta_{nj} \neq 0$ then $\lambda_N(\beta_{ij}) \neq 0$.  
Now 
$$
\lambda_N(u) = [x_n,x_2^N]^{\lambda_N(\beta_{n1})} [x_n,x_2]^{\lambda_n(\beta_{n2})} \ldots [x_n,x_{n-1}]^{\lambda_N(\beta_{n,n-1})}
$$
By Lemma \ref{le:power-in-comm} $[x_n,x_2^N] = [x_n,x_2]^{\frac{a_2^N -1}{a_2-1}}$.  By induction we get
$$
\frac{a_2^N -1}{a_2-1}\lambda_N(\beta_{n1}) = 0, \ \lambda_n(\beta_{n2}) =0, \ \ldots, \ \lambda_N(\beta_{n,n-1}) = 0.
$$
Hence, $\beta_{n1} = 0, \ldots, \beta_{n,n-1} = 0$. This proves uniqueness. 
\end{proof}

\begin{cor} \label{co:normal-forms} Let $X = \{x_1, \ldots,x_n\}$ be a basis of $G$ as a free metabelian group. Then every element $g \in G$ can be uniquely presented as the following product, termed the \emph{normal form} of $g$ relative to $X$:
\begin{equation} \label{eq:normal-form}
    g = x_1^{\gamma_1} \ldots x_n^{\gamma_n}\Pi_{1\leq j<i\leq n}[x_i,x_j]^{\beta_{ij}(a_1,\ldots,a_i)},
\end{equation}
where $\gamma_i \in \Z, 1\leq i\leq n$, $\beta_{ij}(a_1,\ldots,a_i) \in \Z[a_1,a_1^{-1}, \ldots,a_i,a_i^{-1}]$.
\end{cor}
Our next task is to describe the multiplication in $G$  in terms of normal forms. We need some notation. Let $X = \{x_1, \ldots,x_n\}$ be a finite subset of $G$, order $X$ as $x_1< \ldots<x_n$ and form a tuple $x = (x_1, \ldots,x_n)$. Similarly, order the set $C_X = \{[x_i,x_j] \mid 1\leq j <i \leq n \}$, say by introducing the lexicographical order on pairs of indices $(i,j)$,  and form a  tuple $c_x = ([x_2,x_1],[x_3,x_1],  \ldots ,[x_n,x_{n-1}])$. Denote by $\tilde x$ the concatenation $x\cdot c_x = (x_1,\ldots  x_n,[x_2,x_1], \ldots,[x_n,x_{n-1}])$ of $x$ and $c_x$. If $X$ is a basis of $G$ as a free metabelian group  then $\tilde x$ is termed  a \emph{normal form  basis}, or a \emph{module basis}   of $G$, and it's length $n+n(n-1)/2$ is denoted by $dim(G)$.

For $\gamma = (\gamma_1, \ldots,\gamma_n) \in \Z^n$ put 
$$
x^\gamma = x_1^{\gamma_1} \ldots x_n^{\gamma_n},
$$
and for polynomials $\beta_{ij}(a_1,\ldots,a_i) \in \Z[a_1,a_1^{-1}, \ldots,a_i,a_i^{-1}]$, where $a_1 = \bar x_1, \ldots,a_n = \bar x_n$,  form a tuple  $\beta = (\beta_{2,1}, \ldots,\beta_{n,n-1})$ and put 
$$
c_x^\beta  = [x_2,x_1]^{\beta_{2,1}} \ldots, [x_n,x_{n-1}]^{\beta_{n,n-1}}.
$$
If $g \in G$ and $g = x^\gamma c_x^\beta$ then the tuple $t_{\tilde x}(g) = \gamma \cdot \beta$ is called the tuple  of \emph{ coordinates}  (or ${\tilde x}$-coordinates) of $g$ with respect $\tilde x$.  We write $t_{\tilde x}(g) = (t_1(g), \ldots,t_d(g))$, where $d = dim(G)$.    For $g,h \in G$ the multiplication in $G$ completely determines the coordinates $t_{\tilde x}(gh)$ of the product $gh$ in terms of the coordinates $t_{\tilde x}(g)$ and $t_{\tilde x}(h)$. Hence, in this sense $t_i(gh)$ can be viewed as a function of $t_{\tilde x}(g)$ and $t_{\tilde x}(h)$. Our next goal is to describe these functions. 

Denote by $\varepsilon_\ell(z)$ the function $\Z \to \Z\bar G$ defined by $\varepsilon_\ell(z) = \frac{a_\ell^z - 1}{a_\ell -1}$, $\ell = 1, \ldots,n$. Let $\mathcal F$ be the set
 of all formal expressions obtained from variables $V = \{z_1, z_2, \ldots\}$, symbols $0$ and $1$,  and functions $\varepsilon_1, \ldots, \varepsilon_n$ (in the variables from $V$) by finitely many operations of addition $+$ and multiplication $\cdot$. Note that every function $\varepsilon_\ell$, and hence every expression $f(z_1, \dots,z_n) \in \mathcal F$, naturally defines a function $f: R^n \to R$  in every integral domain $R$ (we denote the expression and the corresponding function by the same symbol). Since $R$ is associative, commutative, and unitary, every such function $f$ can be presented in the form 
 \begin{equation} \label{eq:f-p}
 f = p(z_1, \ldots,z_m,\varepsilon_{\ell_1}(y_1), \ldots,\varepsilon_{\ell_s}(y_s)),
 \end{equation}
 where $p = p(z_1, \ldots, z_m, u_1, \ldots,u_s)$ is a polynomial with integer coefficients, and where each variable $u_i$ is replaced by the function $\varepsilon_{\ell_i}(y_i)$ with $1\leq \ell_i \leq n$ and $y_i \in V$.
We are going to prove now that the coordinate functions $t_i(gh)$ are defined by some functions from $\mathcal F$ \emph{uniformly} in the set $X = \{x_1, \ldots,x_n\}$, that is for every $i, 1\leq i\leq d$,  there is a function 
$f_i \in \mathcal F$ (we may assume $f_i$ is in the form (\ref{eq:f-p})) such that for every $n$-element subset $X = \{x_1, \ldots,x_n\} \subseteq G$ ($n$ is fixed upfront)  if  $t_{\tilde x}(g)$ and $t_{\tilde x}(h)$ are coordinates of some elements $g,h \in G$ with respect to $\tilde x$  (which may not be a module basis) then   $t_i(gh) = f_i(t_{\tilde x}(g),t_{\tilde x}(h))$. To prove this we need two technical results.
 
\begin{lemma} \label{le:P} Let $X = \{x_1, \ldots,x_n\}$ be a set of elements of $G$. Then for any $\gamma_1, \ldots, \gamma_n, \delta_1, \ldots,\delta_n \in \Z$ 
          $$x_1^{\gamma _1}\ldots x_n^{\gamma_n}x_1^{\delta _1}\ldots x_n^{\delta_n}=x_1^{\gamma _1+\delta _1}\ldots x_n^{\gamma_n+\delta _n}\overline\Pi,$$
      where 
 $$
 \overline\Pi= \Pi _{j<i}
 [x_i,x_j]^{\frac{(a_i^{\gamma _i}-1)(a_j^{\gamma _j}-1)}{(a_i-1)(a_j-1)}a_{j+1}^{\delta _{j+1}}\ldots a_n^{\delta _n}}.
 $$ 
 \end{lemma}
\begin{proof} In the product 
$$
x_1^{\gamma _1}\ldots x_n^{\gamma_n}x_1^{\delta _1}\ldots x_n^{\delta_n}
$$
we move every $x_j^{\delta_j}$ to the left, using the formulas 
$$
x_i^{\gamma_i}x_j^{\delta_j} = x_j^{\delta_j}x_i^{\gamma_i}[x_i^{\gamma_i},x_j^{\delta_j}],
$$ 
and also the formulas 
$$
[x_i,x_s]^Q x_j^{\delta_j} = x_j^{\delta_j}[x_i,x_s]^{Qa_j^{\delta_j}}
$$
whenever $x_j^{\delta_j}$ meets a commutator $[x_i,x_s]^Q$, $Q \in A$, on its immediate left. We do it  until $x_j^{\delta_j}$ reaches $x_j^{\gamma_j}$. When all $x_j^{\delta_j}$ moved to the left the result will be 
$$
x_1^{\gamma _1+\delta _1}\ldots x_n^{\gamma_n+\delta _n} \overline\Pi= \Pi _{j<i}
 [x_i^{\gamma_i},x_j^{\delta_j}]^{a_{j+1}^{\delta _{j+1}}\ldots a_n^{\delta _n}}
$$
Now the result follows from Lemma \ref{le:power-in-comm}. 

\end{proof}

\begin{lemma} \label{le:P-2} Let $X = \{x_1, \ldots,x_n\}$ be a set of elements of $G$ and elements  $g, h \in G$ be given as products 
$$
g = x_1^{\gamma _1}\ldots x_n^{\gamma_n}\Pi_{1\leq j<i\leq n}[x_i,x_j]^{\beta_{ij}(a_1,\ldots,a_i)}
$$
and 
$$
h= x_1^{\delta _1}\ldots x_n^{\delta_n}\Pi_{1\leq j<i\leq n}[x_i,x_j]^{\nu_{ij}(a_1,\ldots,a_i)}.
$$
Then
\begin{equation} \label{eq:almost-normal}
    gh = x_1^{\gamma _1+\delta _1}\ldots x_n^{\gamma_n+\delta _n}{\overline\Pi}\ \Pi_{1\leq j<i\leq n}[x_i,x_j]^{a_1^{\delta _1}\ldots a_n^{\delta_n}\beta_{ij}(a_1,\ldots,a_i)+\nu_{ij}(a_1,\ldots,a_i)},
\end{equation}
where $\bar\Pi$ is defined in Lemma \ref {le:P}. 

\end{lemma}
\begin{proof}
To prove this, we use the same argument as in Lemma \ref {le:P} and the result itself  from Lemma \ref {le:P}.
\end{proof}

Now we are ready to describe the coordinate functions $t_1(gh), \ldots,t_d(gh)$.

\begin{prop} \label{pr:mult-normal-forms}
     For every $1\leq j < i \leq n$,  there is a function 
$f_{ij} \in \mathcal F$  (we may assume $f_{ij}$ is in the form (\ref{eq:f-p})) such that for any subset $X = \{x_1, \ldots,x_n\}$ of $G$ and for any  $g, h \in G$ given as products 
$$
g = x_1^{\gamma _1}\ldots x_n^{\gamma_n}\Pi_{1\leq j<i\leq n}[x_i,x_j]^{\beta_{ij}(a_1,\ldots,a_i)}
$$
and 
$$
h= x_1^{\delta _1}\ldots x_n^{\delta_n}\Pi_{1\leq j<i\leq n}[x_i,x_j]^{\nu_{ij}(a_1,\ldots,a_i)}.
$$
the following equality holds
\begin{equation} \label{eq:normal}
    gh = x_1^{\gamma _1+\delta _1}\ldots x_n^{\gamma_n+\delta _n} \Pi_{1\leq j<i\leq n}[x_i,x_j]^{f_{ij}(t_{\tilde x}(g),t_{\tilde x}(h))}.
\end{equation}
where $t_{\tilde x}(g)$ and $t_{\tilde x}(h)$ are the coordinates of elements $g,h \in G$ with respect to the tuple  $\tilde x$. 
\end{prop}

\begin{proof} We take the product $gh$ in the form (\ref{eq:almost-normal}). The product there that belongs to $G'$ is not in the normal form yet. We apply inductively equation (\ref{eq:general-Jacoby}) from Lemma \ref{le:general-Jacobi} to this product and bring it to the normal form (\ref{eq:normal})
    
\end{proof}

Now we are ready to give a description of a bases of $G$ as a free metabelian group. 

\begin{theorem} \label{th:normal-forms-char}
A set  $Z = \{z_1, \ldots,z_n \} \subset G$ forms a basis of $G$ if and only if every element $g \in G$ has a unique representation as the following product
$$
g = z_1^{\gamma_1} \ldots z_n^{\gamma_n}\Pi_{1\leq j<i\leq n}[z_i,z_j]^{\beta_{ij}({\bar z}_1,\ldots,{\bar z}_i)},
$$
where $\beta_{ij}({\bar z}_1,\ldots,{\bar z}_i) \in \Z[{\bar z}_1,{\bar z}_1^{-1}, \ldots,{\bar z}_i,{\bar z}_i^{-1}].$
\end{theorem}
\begin{proof} Let $X = \{x_1, \ldots,x_n\}$ be a basis of $G$ as a free metabelian group. Then a map $x_1 \to z_1, \ldots, x_n \to z_n$ extends to a bijection 
$$
\mu:  x_1^{\gamma_1} \ldots x_n^{\gamma_n} \Pi_{1\leq j<i\leq n}[x_i,x_j]^{\beta_{ij}(a_1,\ldots,a_i)}  \to z_1^{\gamma_1} \ldots z_n^{\gamma_n}\Pi_{1\leq j<i\leq n}[z_i,z_j]^{\beta_{ij}({\bar z}_1,\ldots,{\bar z}_i)},
$$
on $G$. This bijection is a homomorphism since by Proposition \ref{pr:mult-normal-forms} the multiplication in $G$ given by the same functions $f_i \in \mathcal F$ in terms of coordinates $t_{\tilde x}$ and $t_{\tilde z}$.
\end{proof}

The following statement follows from the normal forms of elements in $G'$.
\begin{prop}\label{nf} \cite[Proposition 4.4]{KMS} The group $G'$ is a free module over $\mathbb Z[a_1,a_1^{-1},a_2, a_2^{-1}]$ with the basis $\{[x_i,x_j]^{a_3^{\delta _3}\ldots a_{j}^{\delta _j}}\}$ for all $1\leq i<j\leq n$, $\delta _3,\ldots,\delta _j\in\Z^{j-2}.$ 
\end{prop} 

Now we describe briefly the method of Fox derivatives, that we use in the sequent. Let $F_n$ be the free group of rank $n$ with the basis $\{z_1,\ldots ,z_n\},$ let $\pi:F_n\rightarrow\bar g$.  A partial {\em Fox derivative} associated with $z_i$ is the linear map $D_i:\mathbb Z (F_n)\rightarrow \mathbb Z (F_n)$ satisfying the  conditions   
$$D_i(z_i)=1, D_i(z_j)=0, i\neq j$$
and $$D_i(uv)=D_i(u)+uD_i(v) $$ for all $u,v\in F_n$.
The main identity is $D_1(w)(z_1-1)+\ldots +D_n(w)(z_n-1)=w-1.$
$D_i$ induces a linear map  $d_i:\Z G\rightarrow \Z \bar G.$ We briefly explain the details. One can compute
$$D_i([u^{-1},v^{-1}])= (1-uvu^{-1})D_i(u^{-1})+(u-uvu^{-1}v^{-1})D_i(v)$$ for all $u,v\in F_n.$ It follows that for $\pi ':\Z F_n\rightarrow \Z \bar G$,  for all $w\in {F_n}''$, $w\in\ker\pi '$. Hence $D_i$ induces a linear map $d_i:\Z G\rightarrow \Z \bar G$ (that we will also call Fox derivative).

From the definition we have $$d_i(x_i)=1,\ d_i(x_j)=0, i\neq j,$$
$$d_i(uv)=d_i(u)+(u\pi)d_i(v)$$ for all $u,v\in G$. The main identity is $d_1(w)(a_1-1)+\ldots +d_n(w)(a_n-1)=w\pi -1,$ where $w$ is an arbitrary element of $G$.

It can be verified that for $w\in G'$ and $a\in\bar G$, $d_i(w^{a})=a^{-1}d_i(w).$  Also note that for
$w\in G'$ and $u\in G$ we have
$$d_i(wu)=d_i(w)+d_i(u).$$
If $a=\alpha _1a_1+\ldots +\alpha _ka_k\in\Z\bar G$ with $\alpha _i\in \Z, a_i\in\bar G$, we denote by $a_{inv}=\alpha _1{a_1}^{-1}+\ldots +\alpha _k{a_k}^{-1}.$ Then
$$d_i(w^{a})=a_{inv}d_i(w).$$ For $u\in G, \alpha\in\Z$, $$d_i(u^{\alpha})=\frac{u^{\alpha}-1}{u-1}d_i(u).$$

\subsection  {$\Z$ is absolutely interpretable in  $G $}\label{se:3.2}
 
Let $A$ and $B$ be abelian groups, and let $f:A\times A\to B$ be a bilinear map between them. We associate with such $f$ a two-sorted structure $(A,B;f)$ (here $A$ and $B$ are groups, and $f$ is the predicate for the graph of $f$). The map $f$ is said to be \emph{non-degenerate} if for $a \in A$ $f(a,A)=0$ if and only if $a=0$, and similarly, $f(A,a) = 0$ implies $a =0$.  The map $f$ is called \emph{full} if  $B$ is generated by $f(A,A)$. An associative commutative unitary ring  $R$  is a \emph{ring of scalars}  of $f$  if there exist faithful actions of $R$ on $A$ and $B$, which turn $A$ and $B$ into $R$-modules and  such that $f$ is  $R$-bilinear with respect to these actions. There is a canonical embedding of $R$ into  $End(A)$.  $R$ is termed the \emph{largest} ring of scalars of $f$ if for any other ring of scalars $R'$ of $f$, one has $R'\leq R$ when viewed as subrings of $End(A)$. If $f$ is full and non-degenerate then the maximal  ring of scalars of $f$ exists and it is unique \cite{Myasnikov1990}; we denote it  by $R(f)$. 
Moreover, it was shown in  \cite{Myasnikov1990} that if $f:A\times A\to B$ is a full non-degenerate bilinear map between finitely generated abelian groups $A$ and $B$ then the largest ring of scalars $R(f)$ of $f$  and its actions on $A$ and $B$ are absolutely interpretable in the structure $(A,B;f)$. Here, we say that the actions of $R(f)$ on $A$ and  on  $B$ are absolutely interpretable in $(A,B;f)$ if the two-sorted structures $(A,R(f);s_A(x,y,z))$ and  $(B,R(f);s_B(x,y,z))$, where $A,B$ are groups, $R(f)$ is a ring, and $s_A(x,y,z)$ and $s_B(x,y,z)$ are predicates that define the scalar multiplications of $R(f)$ on $A$ and $B$, respectively, are absolutely interpretable in  $(A,B;f)$. We use these facts to describe an absolute interpretation of $\Z$ in $G$.

Note that every verbal subgroup of $G$ has finite verbal width \cite{Rom}, hence it is absolutely definable in $G$ (see, for example, \cite{GMO1}). It follows that all the terms of the lower central series of $G$, in particular, the commutant $G'$ and $G_3 = [G,G']$ are absolutely definable in $G$. 
Hence the free nilpotent group $G/G_3$ is absolutely interpretable in $G$, as well as  the bilinear map
\begin{equation}\label{e: bilinear_map_of_2_nilp}
f_{G}: G/G' \times G/G' \to G'/G_3, 
\end{equation}
the commutation in $G/G_3$,  defined by $(xG',  yG') \mapsto [x,y]$. The map $f_{G}$ is non-degenerate and full, while the abelian groups $G/G'$ and $G'/G_3$ are finitely generated,  hence by the result mentioned above there is a largest ring of scalars $R(f_{G})$ of $f_{G}$,  such that $R(f_G)$, and its actions on $G/G'$ and $G'/G_3$ are absolutely definable in the structure $(G/G',G'/G_3; f_G)$, hence in $G$.  To get an  absolute interpretation of $\Z$ in $G$ it suffices to note that,  as was shown in \cite{GMO1} and earlier by other means in  \cite{MS}, $R(f_{G}) \simeq \Z$.  Thus, $\Z$ and its actions on $G/G'$ and $G'/G_3$ are absolutely interpretable in $G$.  We denote this interpretation of $\Z$ in $G$ by $\Z^*$. We remark that this interpretation of $\Z$ in $G$ is never injective, because the interpretation of the group $G/G_3$ in $G$ is based on  a  non-trivial equivalence relation $mod G_3$ in $G$.  We will describe in Section \ref{subsec:3.5} below another, not absolute, but regular  interpretation of $\Z$ in $G$ which is injective.

Now, we may use in our formulas  expressions of the type $y = x^m \ mod \ G'$ for $x,y \in G \smallsetminus G'$, as well as  $p^m = q \ mod G_3 $ for $p, q \in G'$,  and $m \in \Z$, viewing them as notation for the corresponding formulas of group theory language which are coming from the interpretations of $\Z^*$ and its actions on $G/G'$ and $G'/G_3$.  More precisely, the interpretation $\Z^*$ is given by a definable in $G$ subset $U^* \subseteq G^k$  together  with a definable in $G$ equivalence relation $\sim$ on $U^*$  and formulas $\psi_+(\bar x, \bar y,\bar z), \psi_\circ(\bar x, \bar y,\bar z)$ with $k$-tuples of variables $\bar x, \bar y, \bar z$, which  define binary operations on the factor-set $U^*/\sim$ (denoted by $+$ and $\circ$) and the structure $\langle U^*/\sim; +, \circ\rangle$ is a ring, that is  isomorphic to $\Z$. For $m \in \Z$ by $m^*$ we denote the image of $m$ in $\Z^*$ under the isomorphism $\Z \to \Z^*$. Furthermore, as we mentioned, the exponentiation by $\Z^*$ on $G/G'$ and on $G'/G_3$ is also 0-interpretable, which means that there are formulas in the group language, say $expnil_1(u,v,\bar x)$ and $expnil_2(u,v,\bar x)$, such that  for $g,h \in G$ and $m \in \Z$ one has $g^m = h  (mod \ G')$ if and only if $expnil_1(g,h,m^*)$ holds in $G$ and also for elements $p, q \in G'$  $ p^m  = q (mod \ G_3)$ if and only if $expnil_2(p,q,m^*)$ holds in $G$.

\subsection{Interpretation of $\Z$-exponentiation in $G$} 
\label{se:3.4}

Now, in the notation above,  we construct a formula $exp(u,v,\bar x)$ of the group language, where $\bar x$ is a $k$-tuple of variables, such that for $g,h \in G$ and $m \in \Z$ the following holds
$$
g = h^m \Longleftrightarrow G  \models exp(g,h,m^*).
$$
To construct the formula $exp(u,v,\bar x)$ we consider two cases,  for each of them build the corresponding formula $exp_i(u,v,\bar x)$, and then use them to build $exp(u,v,\bar x)$.

\medskip \noindent
 Case 1. Let $g \in G\smallsetminus G'$.  In Section \ref{se:3.2} we described  a formula  $expnil_1(u,v,\bar x)$ of group language such that   for $g,h \in G$ and $m  \in \Z$  one has 
$$
g^m = h (mod \ G') \Longleftrightarrow  G \models \ expnil_1(g,h,m^*).
$$ 
Now put
$$
exp_1(u,v,\bar x) =  ([u,v] = 1 \wedge expnil_1(u,v,\bar x)).
$$
Then the formula $exp_1(g,h,m^*)$ holds in $G$ on elements $g,h \in G$ and $m^* \in \Z^*$ if and only if $h = g^m (mod \ G') $ and $h \in C_G(g)$. Since the centralizer $C_G(g)$ is cyclic there is only one such $h$  and in this case $h = g^m$.  

\medskip \noindent
 Case 2. Let $1 \neq g \in G'$.  Then as was shown in Section \ref{se:3.1} for any $w \in G \smallsetminus G'$  and every $m \in \Z$ there exists $c \in G'$ such that the following equality holds  
\begin{equation} \label{eq:w-g-m-c}
    (wg)^m  = w^mg^m[c,w].
\end{equation}

Consider the following condition on elements $g,h \in G', m \in \Z$:
\begin{equation} \label{eq:exp-2}
 C_2(g,h,m) =    \forall w (w \in G\smallsetminus G' \rightarrow  \exists c (c \in G' \wedge ((wg)^m  = w^mh[c,w])).
\end{equation}
The equation (\ref{eq:w-g-m-c}) shows that $h = g^m$  satisfies $C_2(g,h,m)$. 
We claim that $h = g^m$ is the only element in $G'$ that satisfies $C_2(g,h,m)$ in $G$.   Indeed, suppose  $C_2(g,h,m)$ holds in $G$ for some $h \in G'$. Then for  any $w \in G\smallsetminus G'$ there exists $c_1 \in G'$ such that 
$$
(wg)^m  = w^mh[c_1,w].
$$
Then $w^mg^m[c,w] = w^mh[c_1,w]$, so 
$$
 h^{-1}g^m= [c,w][c_1,w]^{-1} = [c,w][c_1^{-1},w]  = [cc_1^{-1},w].
$$
Now by Lemma \ref{le:identity}  one gets  $h^{-1}g^m = 1$, so $h = g^m$, as claimed. 
To finish the proof it suffices to show that the condition $C_2(g,h,m)$ can be defined by some  formula $exp_2(u,v,\bar x)$  of group theory in $G$. Note that in $C_2(g,h,m)$ the elements $w$ and $wg$ are in $G \smallsetminus G'$ hence we can use the formula $exp_1(u,v,\bar x) $  to write down the equality $(wg)^m  = w^muh[c,w]$, and then the whole formula 
$exp_2(u,v,\bar x)$. 

Finally, the formula 
$$
\exp(u,v,\bar x) = (u  \notin G' \to exp_1(u,v,\bar x))\wedge (u \in G' \to exp_2(u,v,\bar x))
$$
defines $\Z$-exponentiation on the whole group $G$.

\subsection{Regular injective interpretation of $\Z$ in $G$}
\label{subsec:3.5}

Let  $\exp(u,v,\bar x)$ be the formula from Section \ref{se:3.4}.   Then  for every $g \neq 1$  formula $\exp(g,v,\bar x)$ from Section \ref{se:3.4}  defines a bijection $\lambda_g : \Z^* \to \langle g\rangle$ defined by $m \to g^m$, $m \in \Z^*$.  This bijection allows one to transfer the operations of addition $+$ and multiplication $\circ$ in the ring $\Z^*$ defined in $G$ by the formulas  $\psi_+(\bar x, \bar y,\bar z)$ and $\psi_\circ(\bar x, \bar y,\bar z)$ (see Section \ref{se:3.2}) from  the set $\Z^*$  to the set $\langle g\rangle$.  The resulting definable operations $+_g$ and $\circ_g$ give an interpretation $\Z^*_g = \langle \langle g \rangle; +_g,\circ_g\rangle$ of the ring $\Z$ on the cyclic subgroup $\langle g\rangle$ in $G$ with the coordinate map defined by $\mu_g : g^m \to m$. This interpretation is uniform in $g$, i.e., it has the same formulas for every $1\neq g \in G$, therefore, since the condition $g \neq 1$ is definable in $G$, the interpretations $\Z^*_g$ give a regular interpretation of $\Z$ in $G$, which is injective by construction. 

\subsection{Absolute and regular interpretations of $\Z[a_1,a_1^{-1}, \ldots,a_n,a_n^{-1}]$ in $G$ }
\label{subsec:3.6}

We proved above that the  ring $\Z$ is absolutely interpretable in $G$ as $\Z^*$ (Section \ref{se:3.2}) and also it is regularly injectively interpretable in $G$ via the interpretations $\Z_g^*$, $1 \neq  g \in G$ (Section \ref{subsec:3.5}). By transitivity of interpretations to show that the ring $R = \Z[a_1, a_1^{-1}, \ldots, a_m,a_m^{-1}]$ is absolutely (or regularly injectively) interpretable in  $G$ it suffices to show that $R$  is absolutely interpretable in $\Z$ and then compose this interpretation with the interpretations $\Z^*$ and $\Z_g^*$.

To see that $R$ is absolutely interpretable in $\Z$ note that $R$ is computable (since the word problem in $R$ is decidable), hence there exists a computable  injective function  $\nu: R \to \N $ such that the set $\nu(R)$ is computable in $\N$ and the images under the map $\nu$ of the ring operations of the ring $R$ are computable in $\N$, i.e., the following operations on $\nu(R)$ (here $Q_i \in R$, $i = 1,2,3$) are computable in $\N$:
$$
k_1 \oplus k_2 = k_3 \Longleftrightarrow \wedge_{i = 1}^3 (k_i = \nu(Q_i)) \wedge (Q_1+Q_2 = Q_3),
$$
$$
k_1 \odot k_2 = k_3 \Longleftrightarrow \wedge_{i = 1}^3 (k_i = \nu(Q_i)) \wedge (Q_1\cdot Q_2 = Q_3).
$$
 This is a general fact on computable algebraic structures. Nevertheless, it is convenient to give a sketch of a particular such enumeration $\nu:R \to \N$. Every polynomial $P \in \Z[a_1, \ldots,a_n]$ can be uniquely presented as an integer linear combination of pair-wise distinct monomials on commuting variables $a_1, \ldots,a_n$:
$$
P = \Sigma_{i=1}^d \gamma_i a_1^{\alpha_{i1} }\ldots a_n^{\alpha_{in}},
$$
where $0\neq \gamma_i \in \Z$ and $\alpha_i \in \N$. 
Hence, the polynomial $P$ is uniquely presented by a tuple
\begin{equation} \label{eq:nu(P)}
   u  = (\gamma_1,\alpha_{11}, \ldots\alpha_{1n}, \ldots , \gamma_d,\alpha_{d1}, \ldots, \alpha_{dn})
\end{equation}
If $Q \in R$ then $Q$ can be uniquely presented in the form $Q = \frac{P}{{\bar a}^{\bar \beta}}$ for some $P \in \Z[a_1, \ldots,a_n]$ and some monomial ${\bar a}^{\bar \beta} = a_1^{\beta_1}\ldots a_n^{\beta_n},  \beta_i \in \N$, such that $gcd(P,{\bar a}^{\bar \beta}) = 1$. It follows that $Q$ can be uniquely presented by a pair of tuples $(u_Q,v_Q)$, where $u_Q = u, v_Q = \bar \beta$. Fix an arbitrary computable bijection 
$$
\tau:\bigcup_{i \in \N} \Z^i \to \N
$$
which enumerates all finite tuples of integers. Then $Q$ is uniquely presented by the pair $(\tau(u_Q),\tau(v_Q)) \in \N^2$ and the set of all such pairs is a computable subset of $\N^2$. For a fixed computable bijection $\tau_2:\N^2 \to \N$  put
$$\nu(Q) = \tau_2((\tau(u_Q),\tau(v_Q))).$$
By construction, the subset $\nu(R)$ is computable in $\N$ and  given a number $k \in \nu(R)$ one can algorithmically find the corresponding  Laurent polynomial $Q$ such that $k = \nu(Q)$. Then it is easy to see that the operations $\oplus$ and $\odot$ on $\nu(R)$ are computable in $\N$. To finish the proof it suffices to note that all the computable operations or predicates on $\N$ are definable in $\N$, and $\N$ is definable in $\Z$ (every non-negative integer is a sum of squares of four integers). This shows that $R$ is absolutely interpretable in $\Z$. Hence it is absolutely interpretable in $G$ via $\Z^*$ we denote this interpretation by $R^*$. For $Q \in R$ by $\nu^*(Q) \in R^*$ we denote the image of $\nu(Q)$ under the isomorphism $\Z \to \Z^*$.  Similarly, $R$ is regularly injectively interpretable in $G$ via the interpretations $\Z_g^*, g \neq 1$, which we denote by $R_g^*$. The image of $\nu(Q)$ in $R^*_g$ under the isomorphism $\Z \to \Z^*_g$ we denote by $\nu^*_g(Q)$. Finally, we need to mention the coordinate maps of these interpretations. Let $\bar a = (a_1, \ldots,a_n)$ be an arbitrary fixed  basis of $G$ then the map $\mu_{\bar a}: R^* \to R$ that maps $\nu^*(Q) \in R^* \to Q \in \Z[a_1,a_1^{-1}, \ldots, a_n,a_n^{-1}]$ is the coordinate map for the interpretation $R^*$, and the map $\mu_{\bar a}: R^*_g \to R$ that maps $\nu^*_g(Q) \in R^*_g \to Q \in \Z[a_1,a_1^{-1}, \ldots, a_n,a_n^{-1}]$  is the coordinate map of $R_g^*$.

\subsection{Interpretation of $\Z \bar G$-module $G'$ in $G$}\label{2.4}

In this section for a fixed basis $\bar a = (a_1, \ldots,a_n)$ of $G$ we interpret in $G$ the action of the ring  $R = \Z[a_1,a_1^{-1}, \ldots, a_n,a_n^{-1}]$ on $G'$. More precisely, we describe an interpretation of the $R$-module $G'$ (viewed as a two-sorted structure $G'_R =(G',R;s)$ where  $G'$ is an abelian group, $R$ is a ring, and $s$ is the predicate for scalar multiplication or $R$ on $G'$, see Section \ref{se:3.2}) in $G$ with parameters $\bar a$. In fact, we  interpret $G'_R$ in $G$ as $(G'_R)^* =(G',R^*;s^*)$, where $R^*$ is the interpretation of $R$ in $G$ from Section \ref{subsec:3.6}, and $s^*$ is the  predicate for  the action of $R^*$ on $G'$. We need parameters $\bar a$ to interpret $s^*$. This interpretation is uniform in $\bar a$ (the same formulas work for other bases $\bar b$ of $G$).  A similar  argument gives an injective interpretation of the module $G'_R$ in $G$ as $ (G',R_g^*;s^*)$ in $G$ with parameters $\bar a$ and $g$.

To interpret the module $G'_R =(G',R;s)$  in $G$ for a given basis $\bar a$ of $G$ we interpret $G'$ as $G'$ (which is a definable subgroup of $G$) and the ring $R$ by $R^*$, so  it suffices to show  how to interpret the predicate $s$ in $G$. We need two  preliminary results.

 For a tuple $\bar \alpha  = (\alpha_1, \ldots,\alpha_m) \in \Z^m$, $m\leq n$, denote by $\lambda_{\bar \alpha}$ the homomorphism $\lambda_{\bar \alpha}: \Z[a_1, \ldots,a_n] \to \Z[a_{m+1},\ldots ,a_n]$ such that $a_i \to \alpha_i, i = 1, \ldots,m$.  The kernel $I_{\bar \alpha}$ of $\lambda_{\bar \alpha}$ is the  ideal generated in  $\Z[a_1, \ldots,a_n] $  by  $\{a_1- \alpha_1, \ldots, a_m -\alpha_m\}$.
Notice, that for every polynomial $P = P(a_1, \ldots,a_m) \in \Z[a_1, \ldots,a_n]$ one has $\lambda_{\bar \alpha}(P) = P(\alpha_1, \ldots,\alpha_m)$, so 
$$
P(a_1, \ldots,a_m) = P(\alpha_1, \ldots,\alpha_m)  + \Sigma_{i=1}^m(a_i-\alpha_i)f_i,
$$
for some $f_i \in \Z[a_1, \ldots,a_n]$.

Let $A$ and $B$ be rings and $\Lambda$ a set of homomorphisms from $A$ into $B$.  Recall that $A$ is discriminated into $B$ by a set $\Lambda$ if for any finite subset $A_0 \subseteq A$ there is a homomorphism $\lambda \in \Lambda$ which is injective on $A_0$.  

The following result is known, but we need the proof itself. 

\medskip \noindent
\begin{claim} The ring $\Z[a_1, \ldots,a_n]$ is discriminated into $\Z$ by the set of homomorphisms $\{\lambda_{\bar \alpha} \mid \bar \alpha   \in \Z^n\}$.  
\end{claim}
\begin{proof} Since $\Z[a_1, \ldots,a_n]$  is an integral domain it suffices to show $\Lambda$ separates  $\Z[a_1, \ldots,a_n]$ into $\Z$, i.e., for every non-zero polynomial $Q \in \Z[a_1, \ldots,a_n]$ there exists $\lambda \in \Lambda$ such that $\lambda(Q) \neq 0$. Indeed, let $A_0 = \{P_1, \ldots, P_t\}$ with $P_i \neq P_j$ for $1 \leq j < i \leq t$. Put $Q_{ij} = P_i  - P_j$ and $Q  = \Pi _{1\leq j <i \leq t}Q_{ij}$. Then $Q \neq 0$. If for some $\lambda \in \Lambda$ $\lambda(Q) \neq 0$ then $\lambda$ is injective on $A_0$. 

 Now we prove by  induction on $n$ that $\Lambda$ separates  $\Z[a_1, \ldots,a_n]$ into $Z$.  If $P \in \Z[a_1]$ then $\lambda_{\alpha_1}$ for each sufficiently large $\alpha_1$ separates $P$ into $\Z$.   If $P \in \Z[a_1, \ldots,a_n]$ then for some $m \in \N$ 
 $$
 P = Q_m a_n^m + Q_{m-1}a_n^{m-1} + \ldots +Q_1a_n + Q_0,$$
where $Q_i \in \Z[a_1, \ldots,a_{n-1}]$ and $Q_m \neq 0$. By induction there is $\bar \beta  =(\beta_1, \ldots,\beta_{n-1}) \in \Z^{n-1}$ such that the homomorphism $\lambda_{\bar \beta}$ discriminates $Q_m$ into $\Z$. Then 
$$
\lambda_{\bar \beta}(Q_m)a_n^m + \lambda_{\bar \beta}(Q_{m-1})a_n^{m-1} + \ldots +\lambda_{\bar \beta}(Q_1)a_n + \lambda_{\bar \beta}(Q_0)
$$ 
is a non-zero polynomial in $\Z[a_n]$. Now one  can separate this polynomial into $\Z$ by sending $a_n$ to a large enough integer $\alpha_n$, as above.   This proves the claim.

\end{proof}

For $\bar \alpha = (\alpha_1, \ldots,\alpha_m) \in \Z^m$ denote by $(G')^{I_{\bar \alpha}}$ the submodule of the module  $G'$ obtained from $G'$ by the action of the ideal $I_{\bar \alpha}$.  $(G')^{I_{\bar \alpha}}$ is  an abelian subgroup of $G$  generated by the  set  $\{g^Q \mid g \in G', Q \in I_{\bar \alpha}\}$, hence by the   set $\{g^{a_i - \alpha_i} \mid g \in G', i = 1, \ldots, m\}$.

\medskip \noindent
\begin{claim} For any basis $(a_1, \ldots,a_n)$ of $G$ and any tuple $(\alpha_1, \ldots,\alpha_m) \in \Z^m$ the  subgroup $(G')^{I_{\bar \alpha}} \leq G'$ is definable in $G$ uniformly in $(a_1, \ldots,a_n)$ and $(\alpha_1, \ldots,\alpha_m)$.  More precisely, let $\Z^*$ be 0-interpretation of $\Z$ in $G$ from Section \ref{se:3.2}. Then there is a formula $\varphi(y,y_1,\ldots,y_n,\bar z_1, \ldots, \bar z_m)$ of group theory such that for any basis $(a_1, \ldots,a_n)$ of $G$ and any tuple $(\bar k_1,  \ldots, \bar k_m) \in (\Z^*)^n$  the formula 
$\varphi(y,x_1,\ldots,x_n,\bar k_1, \ldots, \bar k_m)$ defines in $G$ the subgroup $(G')^{I_{\bar \alpha}}$, where $\alpha_i = \bar k_i \in \Z^*, i = 1, \ldots,m$.
\end{claim}
Indeed, the abelian group $(G')^{I_{\bar \alpha}}$  is generated by the   set $\{g^{a_i - \alpha_i} \mid g \in G', i = 1, \ldots, m\}$. 
 It follows that every element $u \in (G')^{I_{\bar \alpha}}$ can be presented as a product 
$$
u = g_1^{a_1-\alpha_1} \ldots g_m^{a_m-\alpha_m},
$$
for some $g_1, \ldots,g_m \in G'$, or , equivalently, 
in the form 
\begin{equation} \label{eq:6.1.1}
u = g_1^{a_1}g_1^{-\alpha_1} \ldots g_m^{a_m}g_m^{-\alpha_m}
\end{equation}
where $g_i^{a_i}$ is a conjugation of $g_i$ by $a_i$, and $g_i^{-\alpha_i}$ is the standard exponentiation of $g_i$ by the integer $-\alpha_i$, $i = 1, \ldots,m$.  It was shown  that there exists a formula $\exp_2(u,v,\bar z)$ such that for any $g, h  \in G'$ and $\alpha = \bar k \in \Z^*$  the formula
$\exp_2(g,h,\bar k)$ holds in $G$ if and only if $g = h^{\alpha}$.    Using formula $\exp_2(u,v,\bar z)$ and definability of the commutant $G'$ in $G$ (see Section \ref{se:3.2}) one can write down the condition (\ref{eq:6.1.1}) by a group theory formula uniformly in $(a_1, \ldots,a_n)$ and $(\alpha_1, \ldots,\alpha_m)$, as claimed.

\begin{lemma} \label{le:6.3}\cite[Lemma 4.24]{KMS}
Let $g, h \in G'$ and $P \in \Z[a_1, \ldots,a_m]$, $m\leq n$. Then $g^P = h$ if and only if the following condition holds:
\begin{equation} \label{eq:g-h-P}
\forall \alpha_1, \ldots \alpha_m \in \Z(g^{P(\alpha_1, \ldots,\alpha_m)} = h \ mod \ (G')^{I_{\bar \alpha}}).
\end{equation}
\end{lemma}
\begin{cor} \label{co:formula-Q}
    Let  $g, h \in G'$  and  $Q =P(a_1,...,a_m)(a_1^{k_1}\ldots a_m^{k_m})^{-1}\in  \Z[a_1,a_1^{-1}, \ldots,a_n,a_n^{-1}]$. Then   $g^Q=h$ if and only if 
    $$
    \exists f \in G' \forall \bar \alpha \in \Z^m (g^{P(\bar \alpha)} = f \ mod \ (G')^{I_{\bar \alpha}}) \wedge (h^{\alpha_1^{k_1} \ldots \alpha_m^{k_m}} = f \ mod \ (G')^{I_{\bar \alpha}}).
    $$
\end{cor}
\begin{lemma} \label{le:action}
    The following is true in $G$:
\begin{enumerate}
\item [1)] There is a formula $E(x,y,\bar u,\bar v)$ of group language such that for any basis $\bar a$ of $G$ for any $g,h \in G'$ and any $Q \in  \Z[a_1, \ldots,a_m]$ one has
$$
G \models E(g,h,\nu^*(Q),\bar a) \Longleftrightarrow g^Q = h;
$$
\item [2)] There is a formula $D(x,y,u,\bar v,z)$ of group language such that for any basis $\bar a$ of $G$, for any $1 \neq f \in G$,  for any $g,h \in G'$ and any $Q \in \Z[a_1, \ldots,a_m]$ one has
$$
G \models D(g,h,\nu_f^*(Q),\bar a, f) \Longleftrightarrow g^Q = h.
$$
\end{enumerate}
\end{lemma} 
\begin{proof}
    We prove 1), the argument for 2) is similar. In view of Lemma \ref{le:6.3} and Corollary \ref{co:formula-Q} it suffices to show that the condition (\ref{eq:g-h-P}) can be written by a formula of group language. By Claim 2 the  subgroup $(G')^{I_{\bar \alpha}} \leq G'$ is definable in $G$ uniformly in $\bar a = (a_1, \ldots,a_n)$ and $\bar \alpha = (\alpha_1, \ldots,\alpha_m)$, hence  the relation $u = v \ mod \ (G')^{I_{\bar \alpha}}$ is definable in $G$ uniformly in $\bar a $ and $\bar \alpha$.  In Section \ref{se:3.4} we showed that the exponentiation in $G$ by elements from $\Z^*$ is definable in $G$ by the formula $\exp(u,v,\bar x)$.  To finish the proof it suffices to show that there is a formula $M(\bar u, \bar u_1, \ldots, \bar u_m,\bar w)$ of group language such that 
    for any $P \in \Z[a_1,a_1^{-1}, \ldots,a_m,a_m^{-1}]$, $\alpha_1, \ldots,\alpha_m \in \Z$, and $\beta \in \Z$  one has 
    \begin{equation} \label{eq:M}
    G \models M(\nu^*(P),\alpha_1^*,\ldots,\alpha_m^*, \beta^*) \Longleftrightarrow P(\alpha_1^*,\ldots,\alpha_m^*) = \beta^*.
    \end{equation}
    Note, that given $\nu^*(P),\alpha_1^*,\ldots,\alpha_m^*$ one can compute in $\Z^*$ the value $P(\alpha_1^*,\ldots,\alpha_m^*)$. Indeed, from $\nu^*(P)$ one can recover by formulas in $\Z^*$ the tuple 
    $$
   u^*  = (\gamma_1^*,\alpha_{11}^*, \ldots\alpha_{1n}^*, \ldots , \gamma_d^*,\alpha_{d1}^*, \ldots, \alpha_{dn}^*)
    $$
   from (\ref{eq:nu(P)}) in Section \ref{subsec:3.6} that describes in $\Z^*$ the polynomial $P$ (since $\Z^* \simeq \Z$ one can use the same formulas that recover $P$ from $u$ in $\Z$). Having $u^*$ and $\alpha_1^*,\ldots,\alpha_m^*$  one can compute in $\Z^*$ the number $P(\alpha_1^*,\ldots,\alpha_m^*)$. Since every computable predicate in $\Z^*$ is definable in $\Z^*$, there exists a formula in the ring language that defines the predicate $P(\alpha_1^*,\ldots,\alpha_m^*) = \beta^*$ in $\Z^*$. But $\Z^*$ is absolutely interpretable in $G$ hence there is the required formula $M(\bar u, \bar u_1, \ldots, \bar u_m,\bar w)$ of group language which satisfies (\ref{eq:M}). This proves 1). 
    
\end{proof}

\begin{prop}
Let $\bar a$ be a basis of $G$. Then  following  are true:
\begin{enumerate}
    \item [1)] The module $G'_R =(G',R;s)$  is interpretable  in $G$ with parameters $\bar a$ as  $(G',R^*;s^*)$, where $R^*$ is the interpretation of $R$ in $G$ from Section \ref{subsec:3.6}, and $s^*$ is the  predicate for  the action of $R^*$ on $G'$ defined by the formula $E(x,y,\bar u,\bar v)$ from Lemma \ref{le:action}. 
    This interpretation is uniform in bases $\bar a$.
    \item [2)] The module $G'_R =(G',R;s)$  is injectively interpretable  in $G$ with parameters $\bar a$ and $1 \neq g \in G$ as  $(G',R_g^*;s^*)$, where $R_g^*$ is the interpretation of $R$ in $G$ from Section \ref{subsec:3.6}, and $s^*$ is the  predicate for  the action of $R^*$ on $G'$ defined by the formula $D(x,y,u,\bar v,z)$ from Lemma \ref{le:action}. 
    This interpretation is uniform in bases $\bar a$ and $1 \neq g \in G$.
    \end{enumerate}
\end{prop}
\begin{proof}
    As we mentioned above $G'$ is absolutely definable in $G$, the ring $R$ is interpretable in $G$ via interpretations $R^*$ and $R^*_g,$ where $ g \in G, g \neq 1$. By Lemma \ref{le:action} the predicates $s^*$ from $(G',R^*;s^*)$ and $(G',R_g^*;s^*)$ are definable in $G$ by formulas $E(x,y,\bar u,\bar v)$ and $D(x,y,u,\bar v,z)$ with parameters $\bar a$ and $\bar a, g $ uniformly in these parameters. This proves the theorem.
\end{proof}

\subsection{Regular bi-interpretability of $\Z$ and $G$}

\begin{theorem}  \label{th4} The following holds in $G$.
\begin{enumerate}
\item [1)]  The set of all free bases $x = (x_1, \ldots,x_n)$  of $G$  is absolutely definable in $G$.
\item [2)] There is a formula $F(u,w,\bar v, \bar y, \bar z)$, where $\bar v$ and $\bar y$ are tuples of variables of length $n$, and $\bar z$ is a tuple of variables of length $n(n-1)/2$  such that for any $h \in G$, any $g \in G, g \neq 1$, any basis $x = (x_1, \ldots,x_n)$ of $G$, any tuple $\gamma \in \Z^n$, and any tuple $\beta = (\beta_{2,1}, \ldots,\beta_{n,n-1})$, where $\beta_{i,j} \in \Z[a_1,a_1^{-1}, \ldots,a_i,a_i^{-1}]$, $a_i = \bar{x_i}$, the following equivalence holds:
$$
G \models F(h,g,x, {\gamma}^*, \nu(\beta)^*)\Longleftrightarrow h = x^\gamma c_x^\beta, 
$$
i.e., $t_{\tilde x}(h) = \gamma\cdot\beta$ (see Section \ref{subsec:2.5}). Here $m \to m^*$ is the isomorphism $\Z \to \Z_g^*$.
\end{enumerate}
\end{theorem}

\begin{proof}
It follows from Theorem \ref{th:normal-forms-char} and  interpretation of $\Z$-exponentiation on $G$ (Section \ref{se:3.4}) and interpretation of $\Z(\bar G)$-exponentiation on $G'$ (Section \ref{2.4}).
    
\end{proof}

\begin{prop} \label{pr:6}
    $G$ is absolutely and injectively interpretable in $\Z$.
\end{prop}
\begin{proof}
 Fix a basis  $\bar x = (x_1, \ldots,x_n)$ in $G$.   
Then any element $g \in G$  can be written in the canonical form (relative to $\bar x$)

\begin{equation}\label{form2}
g=x_1^{\gamma _1}\ldots x_n^{\gamma_n} \Pi_{1\leq j<i\leq n}[x_i,x_j]^{\beta_{ij}(a_1,\ldots,a_i)},
\end{equation}
where $a_i = x_iG' \in G/G'$, $\beta_{ij}(a_1,\ldots,a_i) \in \Z[a_1,a_1^{-1}, \ldots,a_i,a_i^{-1}]$. Hence $g$ 
can be uniquely represented by a tuple of integers \begin{equation}\label{form1}
t(g) = (\gamma _1,\ldots \gamma_n,\nu(\beta _{21}), \ldots ,\nu(\beta _{n,n-1})),
\end{equation}
which are the coordinates of $g$ in  the normal form relative to the basis $\bar x$ (here $\nu(\beta_{ij})$ is the code of the polynomial $\beta_{ij}$ under the computable bijection $\nu$ from Section \ref{subsec:3.6}).

The set $S_G = \{t(g) \mid g \in G\}$ is clearly definable in $\Z^{n+n(n-1)/2}$. For a tuple $\bar s \in S_G$ denote by $\mu_{\bar x}(\bar s)$ the unique element $g \in G$ with  coordinates $t(g)$ with respect to $\bar x$.  This gives a  bijection $\mu: S_G \to G$. 

Observe that the multiplication and inversion in $G$ (as operations on the canonical forms of elements) are computable, since the word problem in $G$ is decidable.   It follows that the multiplication and inversion in $G$ when elements $g$ of $G$ are given through their codes $\nu_G(g)$ are computable in $\N$, therefore their graphs are computably enumerable, and hence definable in arithmetic $\N$, as well as in $\Z$.  This gives an interpretation $\Gamma$ of $G$ in $\Z$, $G \simeq \Gamma(\Z)$,  with the coordinate map $\mu_{\bar x}$. 
\end{proof}

\begin{theorem}  \label{metab} Every free metabelian group $G$ of finite rank $\geq 2$ is regularly strongly and injectively bi-interpretable  with $\mathbb Z$.
\end{theorem} 
\begin{proof}
    By Proposition \ref{pr:6} $G$ is absolutely and injectively interpretable in $\Z$ by some code $\Gamma$, i. e., $G \simeq \Gamma(\Z)$, and with the coordinate map $\mu_x: \Gamma(\Z) \to G$, which depends on a choice of a basis $x$ of $G$. On the other hand $\Z$ is injectively and regularly interpretable as $\Z_g^* = \langle \langle g \rangle; +_g,\circ_g\rangle$, with a parameter $g \in G, g \neq 1$, and the coordinate map $\mu_g: g^m \to m$ (see Section \ref{subsec:3.5}). This gives a regular injective interpretation $\Z \simeq \Delta(G,g)$. 
It follows that $G \simeq \Gamma\circ\Delta(G,g)$ and the coordinate map $\Gamma\circ\Delta(G,g) \to G$ is precisely the map defined by the formula $F(u,g,x, \bar y, \bar z)$ where $g\in G, g \neq 1,$ $x$ is a tuple of parameters (a basis of $G$) that occur in the formula $F(u,w,\bar v, \bar y, \bar z)$  from Theorem \ref{th4} when $g$ is substituted for $w$ and $x$ is substituted for variables $\bar v$. However, the parameters $x$ are not part of the interpretation $\Z \simeq \Delta(G,g)$, to fix this we add $x$ to the set of parameters in $\Delta$, so now $\Z \simeq \Delta(G,g,x)$. Note that the parameters $x$ do not occur in any formulas in $\Delta$ (this is allowed).  By  Theorem \ref{th4} the set of all bases of $G$ is absolutely definable in $G$ , so the interpretation $\Z \simeq \Delta(G,g,x)$ is regular and injective, and the formula $F(u,g,x, \bar y, \bar z)$ defines the coordinate map  $\Gamma\circ\Delta(G,g) \to G$ for parameters $g,x$. In the other direction we have $\Z \simeq \Gamma \circ \Delta(\Z)$, and the corresponding coordinate map $\Gamma \circ \Delta(\Z) \to \Z$ is defined by $(g^*)^m \to m$, where $g^*$ is the image of $g$ under the isomorphism $\mu_x^{-1}:G \to \Gamma(Z)$. The parameter $g^*$ is part of the interpretation $\Z_g \simeq \Delta(\Gamma(\Z),g^*)$. Since $\Gamma(\Z)$ is computable the function $(g^*)^m \to m$ is computable in $\Z$ , hence definable in $\Z$ with parameter $g^*$. 
This proves the theorem.
    
\end{proof}
\begin{cor} Every free metabelian group of finite rank $\geq 2$ has uniform elimination of imaginaries with parameters.
    \end{cor}

\section{Groups elementarily equivalent to a free \\metabelian group} 
In this section we will describe groups elementarily equivalent to a free metabelian group $G$ of rank $n \geq 2$ with basis $x_1,\ldots ,x_n.$

Since $G$ is regularly bi-interpretable  with $\Z$, we can use Theorem \ref{th:main} with $G=\MA$ and $\Z=\MB .$
Then in the notation above  $G=\Gamma (\Z )$ and $\Z=\Delta (G,\bar g).$
If $H\equiv G$, then the same formulas give the interpretation $H=\Gamma (\tilde \Z),$ where $\tilde \Z\equiv\Z$. 
We call $\tilde\N$, a structure elementarily equivalent to $\N$ a model of arithmetic. Notice that $\N$ and $\Z$ are absolutely bi-interprtable. A ring $\tilde\Z$ (resp.  $\tilde\N$) is called a non-standard model if it is not isomorphic to $\Z$ (resp. $\N$), see \cite{Kaye} and \cite{MN}.

First, we notice that $H$ is a so-called exponential group with exponents in $\tilde\Z$. Let us recall the four axioms of exponential groups from \cite{MR1}.
Let $A$ be an arbitrary associative ring with identity and $\Gamma$ a group.
Fix an action of the ring $A$ on $\Gamma$, i.e. a map $\Gamma \times A \rightarrow
\Gamma$. The result of the action of $\alpha \in A$ on $g \in \Gamma$ is written
as $g^\alpha$. Consider the following axioms:

\begin{enumerate}
\item $g^1=g$, $g^0=1$,  $1^\alpha = 1$ ;
\item $g^{\alpha +\beta}=g^\alpha \cdot g^\beta, \ g^{\alpha \beta} =
(g^\alpha)^\beta$;
\item $(h^{-1}gh)^\alpha = h^{-1}g^\alpha h;$
\item $[g,h]=1 \Longrightarrow (gh)^\alpha = g^\alpha h^\alpha.$
\end{enumerate}

\begin{definition}
Groups with $A$-actions satisfying axioms 1)--4) are called {\em $A$-- exponential groups}. 
\end{definition}

These axioms can be written by first-order formulas in $G$ and $H$. This implies the following lemma.

\begin{lemma} The group $H$ is $\tilde Z$-exponential group.

\end{lemma}
Our main goal now is to describe the structure of $H$. We know that $G$ can be represented as a pair $\Z\bar G$ and a module $G'_{\Z \bar G}$ with the action of $\Z\bar G$ on $G'_{\Z \bar G}$ interpretable  in $G$  by Section \ref{2.4}.
%prove the following result.
%\begin{conjecture} $H$ is a free metabelian $\tilde \Z$-group for some $\tilde\Z\equiv \Z$.

%\end{conjecture}

%Notice that for any integral domain $R$ a free metabelian group $G^R$ exists (a universal object in the category of metabelian $R$-groups) and $G\leq G^R$ by \cite{MR2} because $G$ is approximated by torsion-free nilpotent groups.

We have $G\rightarrow _{\Gamma}\Z\rightarrow _{\Delta}G,$ where the interpretation $\Gamma (\Z)$ is via normal forms, therefore  $H\rightarrow _{\Gamma}\tilde\Z\rightarrow _{\Delta}H.$ The element  $g=x_1^{\gamma _1}\ldots x_n^{\gamma_n}u\in G$, where

$$
u = \Pi_{1\leq j<i\leq n}[x_i,x_j]^{\beta_{ij}(a_1,\ldots,a_i)},
$$
where $\beta_{ij}(a_1,\ldots,a_i) \in \Z[a_1,a_1^{-1}, \ldots,a_i,a_i^{-1}] \leq \Z \bar G$, is interpreted as a tuple of elements in $\Z$, $g\rightarrow (\gamma _i,\ldots ,\gamma _n, \bar\beta _{11},\ldots ,\bar\beta _{n-1,n}),$ where $\beta _{ij}$ are tuples.

 For $H$ instead of $\Z\bar G$ we will have  a non standard Laurent polynomial ring  $\tilde\Z\bar G_{NS}$ described in \cite{MN}.  This is a ring elementarily equivalent to $\Z\bar G$. More precisely, regular bi-interpretation of $G$ with $\Z$ induces the regular bi-interpretation of  $\Z\bar G$  with $\Z$, $\Z\bar G=\Gamma _1(\Z)$. Then $\tilde\Z\bar G_{NS}= \Gamma _1(\tilde\Z).$
  The same formula as in the standard case says that for $h\in H$, 
\begin{equation}\label{NF} h=x_1^{\tilde\gamma _1}\ldots x_n^{\tilde\gamma_n}u,\\ \ \ \ 
u = \Pi_{1\leq j<i\leq n}[x_i,x_j]^{\beta_{ij}(a_1,\ldots,a_i)},
\end{equation}
where $\tilde\gamma_i\in\tilde\Z ,$ $\beta_{ij}(a_1,\ldots,a_i) \in \tilde\Z \bar G_{NS}$.   It is interpreted as a non-standard tuple of tuples of elements in $\tilde\Z$, $$g\rightarrow (\tilde\gamma _i,\ldots ,\tilde\gamma _n, \bar\beta _{11},\ldots ,\bar\beta _{n-1,n} ).$$

There is a formula connecting $P[a_1,\ldots ,a_n]\in  \tilde\Z \bar G_{NS}$  and its evaluation $P[\alpha _1,\ldots ,\alpha _n]$, where $\alpha _1,\ldots ,\alpha _n\in \tilde\Z$.
Lemma \ref{le:6.3} gives the interpretation of the action of the standard Laurent polynomial ring $\Z\bar G$ on $G'$ and, therefore, the interpretation of the action of the non-standard Laurent polynomial ring  $\tilde\Z\bar G_{NS}$ on $H'$, where $H'$ is the $\tilde\Z\bar G_{NS}$ module generated by $\{[x_i,x_j]\}$.

 %The following  three lemmas show how an arbitrary element in $H$ can be written in the normal form.
Lemma \ref{le:power-in-comm} implies the following.
\begin{lemma} \label{L6} For any $x,z\in H$ and $\delta, \gamma\in\tilde{\mathbb Z},$ $$[z^\gamma,x^{\delta}]=[z,x]^{(\frac{{\bar z}^{\gamma}-1}{\bar z -1}) (\frac{{\bar x}^{\delta}-1}{\barx -1})},$$
where  $(\bar x^{\delta}-1)/(\bar x-1), (\bar z^{\gamma}-1)/(\bar z-1)\in \tilde\Z \bar G_{NS}$ .

\end{lemma}

 Denote by $a_i$ the image of $x_i\in H$ in $\tilde\Z \bar G_{NS}$. 
%Notice that one  can rewrite $$[x_i,x_j]^{{a_{k}^{\delta} -1}},$$
% where $j<i<k,$ using the Jacobi identity as $$[x_i,{x_k}^{\delta}]^{-(a_j-1)} 
%[x_j,{x_k}^{\delta}]^{a_i-1}.$$ 

% \begin{lemma}\label{le:10} 
%    $x_1^{\gamma _1}\ldots x_n^{\gamma_n}x_1^{\delta _1}\ldots x_n^{\delta_n}=x_1^{\gamma _1+\delta _1}\ldots x_n^{\gamma_n+\delta _n}\overline\Pi,$ where $\overline\Pi\in H'$,  $$\overline\Pi= \Pi _{j<i}
 %[x_i,x_j]^{\frac{(a_i^{\gamma _i}-1)(a_j^{\gamma _j}-1)}{(a_i-1)(a_j-1)}a_{j+1}^{\delta _{j+1}}\ldots a_n^{\delta _n}}.$$ Here $\overline\Pi$ is not in the normal form, but can be transformed to the normal form  using the Jacobi identity and Lemma \ref{L6}.
% \end{lemma}
%\begin{proof} Follows from Proposition \ref{le:P}.
%\end{proof}
%We summarise this now. 
\begin{theorem} If $H$ is a group elementarily equivalent to $G$. Then $H$ has the following structure. \begin{enumerate} \item Elements $h\in H$ have the normal form  
$$ h=x_1^{\tilde\gamma _1}\ldots x_n^{\tilde\gamma_n}u,\\ \ \ \ 
u = \Pi_{1\leq j<i\leq n}[x_i,x_j]^{\beta_{ij}(a_1,\ldots,a_i)},$$
where $\tilde\gamma_i\in\tilde\Z ,$ $\beta_{ij}(a_1,\ldots,a_i) \in \tilde\Z \bar G_{NS}$.   
\item $H'$ is a module over $\tilde\Z\bar G_{NS}$ with generators $\{[x_i,x_j]\}.$
\item Multiplication in $H$ is defined as follows. if  $g, h \in H$ are given by their normal forms
$$
g = x_1^{\tilde\gamma _1}\ldots x_n^{\tilde\gamma_n}\Pi_{1\leq j<i\leq n}[x_i,x_j]^{\beta_{ij}(a_1,\ldots,a_i)}
$$
and 
$$
h= x_1^{\tilde\delta _1}\ldots x_n^{\tilde\delta_n}\Pi_{1\leq j<i\leq n}[x_i,x_j]^{\nu_{ij}(a_1,\ldots,a_i)}.
$$
Then
\begin{equation} \label{eq:almost-normalNS}
    gh = x_1^{\tilde\gamma _1+\tilde\delta _1}\ldots x_n^{\tilde\gamma_n+\tilde\delta _n} \Pi_{1\leq j<i\leq n}[x_i,x_j]^{f_{ij}(t(g),t(h))}.
\end{equation}
where $f_{ij}\in\mathcal F$ are non-standard functions corresponding to functions from Proposition \ref{pr:mult-normal-forms}, $t(g)$ and $t(h)$ are coordinates in $\tilde\Z\bar G_{NS}$ of elements $g,h \in H$ in the base $x\cdot y$. 

\end{enumerate}
\end{theorem}
\begin{proof} The functions 
$\varepsilon_\ell(z)$ and all functions in  $\mathcal F$  are defined in  the ring $\tilde\Z\bar G_{NS}$ by the same formulas as in $\mathbb Z\bar G.$ 
\end{proof}
\begin{lemma} If $x,y\in H-H'$ and $b$ is the image of $y$ in $\tilde Z\bar G_{NS}$, then we have the formula for the $\delta$-commutator $$y^{-\delta}x^{-\delta}(xy)^{\delta}=[x,y]^{-f(a,b)},$$ where $f(a,b)$ is a non-standard polynomial such that $$(a-1)f(a,b)=$$ $$(a^{\delta}b^{\delta}-1)/(ab-1)+(1-b^{\delta})/(b-1)=
 b^{\delta -1}(a^{\delta -1}-1)+b^{\delta -2}(a^{\delta -2}-1)+\ldots +b(a-1).$$
    
\end{lemma}
\begin{proof}
 We can write  that   $y^{-\delta}x^{-\delta}(xy)^{\delta}$ is in the module generated by the  commutator $[x,y]$, and we wish to see what is the non-standard polynomial $f(a,b).$

We have $[x,y^{-\delta}x^{-\delta}(xy)^{\delta}]=[x,y]^{f(a,b)(a-1)}.$ At the same time $$[x,y^{-\delta}x^{-\delta}(xy)^{\delta}]=[x,(xy)^{\delta}][x,y^{-\delta}x^{-\delta}]^{(ab)^{\delta}}=[x,xy]^{((ab)^{\delta}-1)/(ab-1))}[x,y^{-\delta}]^{a^{-\delta}(ab)^{\delta}}=$$
 $$[x,y]^{(a^{\delta}b^{\delta}-1)/(ab-1)+(1-b^{\delta})/(b-1)}.$$ 
 Polynomial $$(a^{\delta}b^{\delta}-1)/(ab-1)+(1-b^{\delta})/(b-1)=
 b^{\delta -1}(a^{\delta -1}-1)+b^{\delta -2}(a^{\delta -2}-1)+\ldots +b(a-1)$$
 is divisible by $(a-1)$ in the standard ring of Laurent polynomials, therefore it is divisible by $(a-1)$ in the non-standard ring. Since the rings are integral domains, $f(a,b)$ is the result of this division.
\end{proof}
\section{$A$-metabelian groups}

The question about varieties of exponential groups was discussed in \cite{AMN}. Let $\Gamma$ be an arbitrary exponential group with exponents in $A$. We set
$$(\Gamma ,\Gamma )_A=\langle (g,h)_{\alpha}=h^{-\alpha}g^{-\alpha}(gh)^{\alpha}, g,h\in\Gamma ,\alpha\in A\rangle _A.$$ The $A$-subgroup $(\Gamma ,\Gamma )_A$ is called the $A$-commutant of the group $\Gamma$.
By \cite{AMN}, a free abelian $A$-group with base $X$ is a free $A$-module and is $A$-isomorphic to the
factor-group of the free $A$-group with base $X$ by its $A$-commutant. The $A$-commutant is called the first $A$ commutant and denoted by $\Gamma ^{(1,A)}$. The $A$-commutant of $\Gamma ^{(1,A)}$ is the second $A$ commutant $\Gamma ^{(2,A)}.$  Then $\Gamma$ is called in \cite{AMN} $n$-step $A$-solvable group if $\Gamma ^{(n,A)}=1.$  Clearly $n$-step $A$-solvable group is $n$-solvable $A$-group. If $\Gamma ^{(2,A)}=1$ we call $\Gamma$ an $A$-metabelian group.  Notice, that $H$ that is elementarily equivalent to a free metabelian group is $\tilde \Z$-metabelian because $\alpha$-commutators belong to $H'$ and commute in $H$.

A  discretely ordered ring is an ordered ring in which there is no element between 0 and 1. 
Let $A$ be a discretely ordered ring and $K$ be a multiplicative $A$-module with generators $a_1,\ldots ,a_n$.
Consider a group algebra $A(K)$. Let $R$ be the $A$ algebra generated by $A(K)$ and for all positive $\delta\in A$, by  series $(a^{\delta}-1)/(a-1)=\Sigma _{0\leq \alpha <\delta} a^{\alpha},$ and
$\Sigma _{0\leq \alpha <\delta}b^{\alpha}\frac{a^{\alpha}-1}{a-1},$  where $a,b\in K.$ 

We define an $A$-metabelian exponential group $M$ with generators $x_1,\ldots ,x_n$ by the following axioms:

1) $M$ is an $A$-metabelian  group.

%Let $M'$ consist of elements that commute with any $\alpha$-commutator.

%2)  For any $\delta \in A$, $y,x \in M$,  $y^{-\delta}x^{-\delta}(xy)^{\delta}$  is in $M'$. 
%Moreover, $M^{(2,A)}=1$
%(therefore $M$ is $A$-metabelian in the Amaglobeli's definition above). 

2)The $A$-commutant $M'$ is an $R$-module.
%3) For any $[z_1,z_2]$ and any $\sigma\in R$ there is  $[z_1,z_2]^{\sigma}$. 

3)   For any $z,x\in M$ and $\delta\in A,$ $[z,x^{\delta}]=[z,x]^{(a^{\delta}-1)/(a-1)}.$ 

4)    For any $x,y\in M$ and $\delta\in A,$ $y^{-\delta}x^{-\delta}(xy)^{\delta}=[x,y]^{f(a,b)},$ where   $$f(a,b)=[(a^{\delta}b^{\delta}-1)/(ab-1)+(1-b^{\delta})/(b-1)]/(1-a)\in R.$$

Let now $M$ be a free group with generators $x_1,\ldots ,x_n$ in the category of $A$-metabelian exponential groups.    

\begin{lemma} The group $M'$ is an  $R$-module generated by  elements $[x_i,x_j]$.   If $u$ belongs to $M'$, then it can be uniquely written as$$
u = \Pi_{1\leq j<i\leq n}[x_i,x_j]^{\beta_{ij}(a_1,\ldots,a_i)},
$$ where $\beta_{ij}(a_1,\ldots,a_i)\in R$.
\end{lemma}
\begin{proof} Consider a $\delta$-commutator $y^{-\delta}x^{-\delta}(xy)^{\delta}$. We have, using identities 1) and 2), $$[x,y^{-\delta}x^{-\delta}(xy)^{\delta}]=[x,(xy)^{\delta}][x,y^{-\delta}x^{-\delta}]^{(ab)^{\delta}}=[x,xy]^{((ab)^{\delta}-1)/(ab-1)))}[x,y^{-\delta}]^{a^{-\delta}(ab)^{\delta}}=$$
 $$[x,y]^{(a^{\delta}b^{\delta}-1)/(ab-1)+(1-b^{\delta})/(b-1)}.$$ 

Every commutator can be represented as
$[x_i^{\beta}, (x_{i_1}\ldots x_{i_t})^{\alpha}, x_{j_1}^{\alpha _1},\ldots ,x_{j_k}^{\alpha _k}]$. We can assume that
$i\geq j_1\geq j_2\ldots \geq j_k$ otherwise use the Jacobi identity. If $i$ is greater than or equal to all $i_1,\ldots,i_t$, then 
the representation from Lemma \ref{L6} gives elements $[x_i,x_{i_1}\ldots x_{i_t}]^{f(a_{i_1},\ldots, a_{i_t})}$. Bringing the commutator $[x_i,x_{i_1}\ldots x_{i_t}]$ to normal form and acting by $f(a_{i_1},\ldots, a_{i_t})$ gives elements in normal form.

Suppose now some of $i_1,\ldots ,i_t$ is greater than $i$. Consider a general example
$$[x_1,(x_2x_3)^{\delta}]=[x_1,(x_2x_3)]^{((a_2a_3)^{\delta}-1)/(a_2a_3-1)}=[x_1,x_2x_3]^{\delta}\Pi _{0\leq\alpha<\delta}[x_1,x_2x_3,(x_2x_3)^{\alpha}]=$$
$$[x_1,x_2x_3]^{\delta}\Pi _{0\leq\alpha<\delta}[x_1,x_2x_3,x_3^{\alpha}]^{x_2^{\alpha}}\Pi _{0\leq\alpha<\delta}[x_1,x_2x_3,x_2^{\alpha}]=$$
$$[x_1,x_2x_3]^{\delta} [x_1,x_2x_3,x_3]^{\Sigma _{0\leq\alpha<\delta} a_2^{\alpha}(a_3^{\alpha}-1)/(a_3-1)}[x_1,x_2x_3,x_2]^{\Sigma _{0\leq\alpha<\delta} (a_2^{\alpha}-1)/(a_2-1)}.$$
Using the Jacobi identity we rewrite $[x_1,x_2x_3,x_3]=[x_3,x_2x_3]^{a_1-1}[x_3,x_1]^{-a_2a_3}.$ This finally gives a representation of  the commutator
$[x_1,(x_2x_3)^{\delta}]$ in the normal form. A general case can be treated similarly.

To prove uniqueness   of normal forms we need the analogue of  Fox derivatives from $A(M)$ to $R$. We define a partial  Fox derivative as a linear mapping $d_i: A(M)\rightarrow R$ satisfying the  properties of $d_i$ from Section \ref{se:3.1} and
\begin{equation}
d_i(g^{\delta})=\frac{g^{\delta}-1}{g-1}d_i(g)=d_i(g)\Sigma _{0\leq\alpha <\delta}g^{\alpha}.\end{equation}

A consequence of these is an  equality:
$$Dg^{-\delta}=-g^{-\delta}Dg^{\delta}.$$

One can also compute for $f(a,b)\in R$ \begin{equation} d_i([x,y]^{f(a,b)})=f(a,b)_{inv}d_i([x,y]).
\end{equation}

The uniqueness of the normal form can be proved by using Fox derivatives and the homomorphisms $\varepsilon _I,$ where $I\subseteq \{1,\ldots ,n\}$ and $x_i\varepsilon _j=x_i $ if $i\in I$ and  $x_i\varepsilon _j=1$ if $i\not\in I$ (as it is done in  \cite{MRom} for normal forms in a free metabelian group). For example, we have $u\varepsilon_{\{1,2\}}=[x_2,x_1]^{\beta _{12}(a_1,a_2)},$ hence ${\beta _{12}(a_2,a_1)}$ is defined uniquely. 
Multiply $u$ by $[x_2,x_1]^{-\beta _{21}(a_1,a_2)}$ to get $u'$.

Then $u'\varepsilon _{\{1,2,3\}}=[x_3,x_1]^{\beta _{31}(a_1,a_2,a_3)}[x_3,x_2]^{\beta _{32}(a_1,a_2,a_3)}.$

Then $d_1(u'\varepsilon _{\{1,2,3\}})=\beta _{31(inv)}(a_1,a_2,a_3)(a_3-1)a_1^{-1}a_3^{-1},$
$d_2(u'\varepsilon _{\{1,2,3\}})=\beta _{32(inv)}(a_1,a_2,a_3)(a_3-1)a_2^{-1}a_3^{-1}.$

This allows to compute uniquely $\beta _{31}(a_1,a_2,a_3)$ and $\beta _{32}(a_1,a_2,a_3),$ respectively, and so on.

\end{proof}

\begin{theorem} If $H\equiv G$, where $G$ is a free metabelian group, then $H$ contains a free $\tilde {\mathbb Z}$-metabelian exponential group as a subgroup.  
\end{theorem}
\begin{proof}

We know that $H$ is $\tilde {\mathbb Z}$-exponential group for some $\tilde {\mathbb Z}\equiv {\mathbb Z}.$. And $H'$  is a $\tilde\Z\bar G_{NS}$-module.  Let $M$ be a free $\tilde\Z$-metabelian exponential group defined above. Then normal forms of elements in $M$ are exactly normal forms  (\ref {NF}) of elements in $H$,  therefore $M$ is a subgroup of $H.$
\end{proof}

%------
% Insert acknowledgments and information
% regarding funding at the end of the last
% section, i.e., right before the bibliography.
%------

%\begin{ack}
%We thank X.
%\end{ack}

%\begin{funding}
%This work was partially supported %by~\ldots
%\end{funding}

%------
% Insert the bibliography.
%------


\begin{thebibliography}{99}\bibitem{AMN} M. G. Amaglobeli, A. G. Miasnikov, T. T. Nadiradze, Varieties of exponential R-groups, Algebra and Logic, 2024, Volume 62, pages 119-136.
\bibitem{AKNS} M. Aschenbrenner, A. Khelif,  E. Naziazeno, T. Scanlon, The logical complexity of finitely generated commutative rings. Int. Math. Res. Not. IMRN 2020, no. 1, 112-166.

\bibitem{Bahmuth} S. Bachmuth, Automorphisms of free metabelian groups, Trans. Amer. Math. Soc. 118 (1965), 93–1104.
%\bibitem{B1} J. Barwise,  {\it Admissible sets and structures}, Springer-Verlag, Berlin-Heidelberg-New York, 1975.

%\bibitem{B2} J. Barwise,  {\it An Introduction to First-Order Logic}, in Barwise, Jon, ed. (1982). Handbook of Mathematical Logic. Studies in Logic and the Foundations of Mathematics. Amsterdam, NL: North-Holland. 


%\bibitem{bauval} A. Bauval, {\it Polynomial rings and weak second order logic}, J. symb. Logic, 50, 1985, 953-972.

%\bibitem{E} Yu.L. Ershov, {\it Definability and computability}, New York: Plenum, 1996.

%\bibitem{Ershov2} Yu.L.  Ershov, V.G. Puzarenko, and A.I. Stukachev, {\it HF-Computability},  in: Computability in Context: Computation and Logic in the Real World, S. B. Cooper and A. Sorbi (eds.), Imperial College Press/World Sci., London, 2011, pp. 169-242.

\bibitem{DM1} E.\,Daniyarova, A.\,Myasnikov, Theory of interpretations I. Foundations, Int. J. of Algebra and Computation, accepted, see also Math Arxiv.CV34.

\bibitem{DM2} E. Daniyarova, A. Myasnikov, Theory of Interpretations II. Categorical equivalence, Math arxiv, 2025.

\bibitem{GMO1}  A. Garreta, A. Miasnikov, D. Ovchinnikov, Diophantine problems in solvable groups, Bulletin of Mathematical Sciences,  Vol. 10, No. 1 (2020), 27 pages.
DOI: 10.1142/S1664360720500058.

\bibitem{GMO2} A. Garreta, A. Miasnikov, D. Ovchinnikov, , Studying the Diophantine problem in finitely generated rings and algebras via bilinear maps, arXiv:1805.02573v2 [math.RA].



\bibitem{Hodges} W. Hodges, Model theory, Cambridge University Press, 1993.



\bibitem{KMga} O. Kharlampovich, A. Myasnikov, What does a group algebra of a free group ``know" about the group?,    Annals of Pure and Applied Logic, 169 (2018) 523-547.

\bibitem{KMS} O. Kharlampovich, A. Myasnikov, M. Sohrabi,  Rich groups, weak second order logic and applications,  De Gruyter, Groups and Model Theory, GAGTA book 2, 2021, 127-193.
\bibitem{Khelif}  A. Khelif, Bi-interpretabilit\'e et structures QFA: \'etude des groupes r\'esolubles et
des anneaux commutatifs, C. R. Acad. Sci. Paris, Ser. I 345 (2007) 59-61; and Preuves (unpublished).
\bibitem{Kaye} R. Kaye, Models of Peano arithmetic, Clarendon Press, 1991.
\bibitem {Mal1} A. Malcev, On free solvable groups, Soviet Math. Doklady 1 (1960), 65-68.
\bibitem{Miasn} A.G. Myasnikov, The structure of models and a decidability criterium for complete theories of finite dimensional algebras,  Math USSR Izvestia, 34, 2,1990.
\bibitem{Miasn1} A. G. Myasnikov,  Definable invariants and abstract isomorphisms of bilinear mappings, Vychisl. Center Siberian Otdel AN SSSR, Novosibirsk, 1984. 

\bibitem{Myasnikov1990} A. G. Myasnikov,  Definable invariants of bilinear mappings, Siberian Math. Journal, 31(1), 1990, 89-99.

\bibitem{MS} A. G. Myasnikov, M. Sohrabi, Groups elementarily equivalent to a free 2-nilpotent group of finite rank, Algebra and Logic, Vol. 48, No. 2, March 2009.
%\bibitem{MS2019} A. G. Miasnikov, M. Sohrabi, Bi-interpretability with $\mathbb Z$ and models of the complete elementary theories of $\text{SL}_n(\mathcal{O})$, $\text{GL}_n(\mathcal{O})$,$\text{T}_n(\mathcal{O})$, $n\geq 3$, Arxiv: 2004.03585.
\bibitem{MR1} A. G. Myasnikov, V.N. Remeslennikov,  Exponential groups, I. Foundations of the theory
and tensor completions, Sib. Mat. Zh., 35, No. 5, (1994) 1106-1118.

\bibitem{MR2} A. Myasnikov and V. Remeslennikov, Exponential groups 2: Extensions of centralizers and tensor completion of CSA groups, 
 International Journal of Algebra and Computation, Vol. 06, No. 06, pp. 687-711 (1996).

\bibitem{MN} A. Miasnikov, A. Nikolaev, Nonstandard polynomials: algebraic properties and elementary equivalence, 2024, 	arXiv:2409.14467 [math.LO], 
https://doi.org/10.48550/arXiv.2409.14467.
%\bibitem{Nies2007} A. Nies, Describing groups, Bull. Sym. Logic, 13 (3) (2007)  306-339.
\bibitem{MRom}A. Myasnikov, V. Roman'kov
Diophantine cryptography in free metabelian groups:
Theoretical base, Groups Complex. Cryptol. 2014; 6 (2):103-120.

%\bibitem{Nies_blog} A. Nies (ed.), Logic Blog, 2015, Part 7, Section 14, available at
%http://arxiv.org/abs/1602.04432.

\bibitem {Rom} V. Romankov, On the width of verbal subgroups of solvable groups, Algebra and Logic,
1982, v. 21, no 1,  60-72.

%------ Example for a paper in journal:
% \bibitem{article1}
% A.~Petrunin, Parallel transportation for Alexandrov space with curvature bounded below.
% \emph{Geom. Funct. Anal.} \textbf{8} (1998), no.~1, 123--148
% \Zbl{0903.53045} \MR{1601854}

%------ Example for a book:
% \bibitem{book1}
% W.~P. Ziemer, \emph{Weakly differentiable functions}.
% Grad. Texts in Math. 120,  Springer, New York, 1989
%\Zbl{0692.46022} \MR{1014685}

%------ Example for a paper in a book:
% \bibitem{incollection1}
% J.~S. Milne, Introduction to Shimura varieties.
% In \emph{Harmonic analysis, the trace formula, and Shimura varieties},
% pp. 265--378, Clay Math. Proc. 4,
% American Mathematical Society, Providence, RI, 2005
% \Zbl{1148.14011} \MR{2192012}

%------ Example for a preprint on arXiv:
% \bibitem{preprint1}
% D.~V. Nguyen, S.~K. Chilappagari, M.~W. Marcellin, and B.~Vasic,
% LDPC codes from latin squares free of small trapping sets.
% 2010, \arxiv{1008.4177}

%------ Example for a report:
% \bibitem{report1}
% J.~Schöberl, Commuting quasi-interpolation operators.
% Technical report isc-01-10-math, Texas A\&M University, 2001,
% \url{www.isc.tamu.edu/publications-reports/tr/0110.pdf}

%------ Example for a thesis:
% \bibitem{thesis1}
% E.~Giorgi, \emph{The geometric universe}.
% Ph.D. thesis, University of Maryland, College Park, 2002

\end{thebibliography}
\end{document}